\documentclass{amsart}
\usepackage{amssymb,amsfonts,amsmath,amsthm,cite,verbatim,mathrsfs,amscd}
\usepackage{xcolor}
\usepackage{tikz}
\usepackage{graphicx}
\usepackage[all,pdf]{xy}
\usepackage{setspace}
\usepackage[left=2.5cm, right=2.5cm, top=3.5cm, bottom=3.5cm]{geometry}

\usepackage{hyperref}

\usepackage[normalem]{ulem}
\newtheorem{theorem}{Theorem}[section]
\newtheorem{thm}[theorem]{Theorem}
\newtheorem{prop}[theorem]{Proposition}
\newtheorem{lem}[theorem]{Lemma}

\newtheorem{cor}[theorem]{Corollary}

\makeatletter \@addtoreset{equation}{section}

\newcommand{\qbinom}[2]{\genfrac{[}{]}{0pt}{}{#1}{#2}}

\newcommand{\Mid}{\:|\:}  
\DeclareMathOperator*{\CT}{CT}

\def\x{\boldsymbol{x}}

\begin{document}

\title[$q$-Baker--Forrester conjecture]{A proof of the multi-component $q$-Baker--Forrester conjecture}

\author{Yue Zhou}

\address{School of Mathematics and Statistics, HNP-LAMA, Central South University,
Changsha 410083, P.R. China}

\email{zhouyue@csu.edu.cn}


\date{March 23, 2025}

\begin{abstract}
\noindent
The Selberg integral, an $n$-dimensional generalization of the Euler beta integral,
plays a central role in random matrix theory, Calogero--Sutherland quantum many body systems,
Knizhnik--Zamolodchikov equations, and multivariable orthogonal polynomial theory.
The Selberg integral is known to be equivalent to the Morris constant term identity.
In 1998, Baker and Forrester conjectured a  $(p+1)$-component generalization of
the $q$-Morris identity. It in turn yields a generalization of the Selberg integral.
The $p=1$ case of Baker and Forrester's conjecture was proved by K\'{a}rolyi, Nagy, Petrov and Volkov in 2015.
In this paper, we give a proof of the $(p+1)$-component $q$-Baker--Forrester conjecture, thereby
settling this 26-year-old conjecture.


\end{abstract}
\maketitle

\section{Introduction}\label{s-intr}

In this introduction, we first briefly review the main developments in the area of constant term identities.
The study of constant term identities dated back to Freeman Dyson, who conjectured his now famous constant term (ex)-conjecture \cite{dyson1962} in 1962.

\begin{thm}[Dyson's ex-conjecture]
For nonnegative integers $a_1,\dots,a_n$,
\begin{equation}\label{thm-Dyson}
\CT_{\x} \,\prod_{1\leq i\neq j\leq n}(1-x_i/x_j)^{a_i}=\frac{(a_1+\cdots+a_n)!}{a_1!\cdots a_n!},
\end{equation}
where $\displaystyle\CT_{\x}f$ denotes taking the constant term of the Laurent polynomial $f$ with respect to $\x:=(x_1,\dots,x_n)$.
\end{thm}
Dyson's conjecture was soon proved by Gunson \cite{gunson} and Wilson \cite{wilson}. Twenty years later Wilson
received his Nobel Prize in physics, while \cite{wilson} is his first publication.
In 1970, using Lagrange interpolation Good \cite{good} discovered a very elegant one-page proof. Much later, Zeilberger gave the first  combinatorial proof using tournaments \cite{zeil}.

In 1975, Andrews \cite[Page 216]{andrews1975} conjectured a $q$-analogue of the Dyson constant term identity.
\begin{thm}[Andrews' ex-conjecture]
For nonnegative integers $a_1,\dots,a_n$,
\begin{equation}\label{q-Dyson}
\CT_{\x} \prod_{1\leq i<j\leq n}
(x_i/x_j)_{a_i}(qx_j/x_i)_{a_j}=
\frac{(q)_{a_1+\cdots+a_n}}{(q)_{a_1}(q)_{a_2}\cdots(q)_{a_n}},
\end{equation}
where $(a)_z=(a;q)_z:=\prod_{j=0}^{z-1}(1-aq^j)$ is the $q$-shifted factorial for a positive integer $z$ and $(a)_0=1$.
\end{thm}
Andrews' conjecture remained open for ten years until Zeilberger and Bressoud \cite{zeil-bres1985} generalized the above-mentioned combinatorial method from \cite{zeil}.
In 2006, Gessel and Xin \cite{GX} obtained a second proof (the first algebraic proof) using iterated Laurent series.
In 2014, K\'{a}rolyi and Nagy \cite{KN} discovered the ``Proof from the book" using the combinatorial Nullstellensatz.
In the same year, Cai \cite{cai} found the first inductive proof by adding an extra parameter to the original problem.

The Laurent polynomials on the left-hand side of \eqref{thm-Dyson} and \eqref{q-Dyson} are usually called the Dyson product and the $q$-Dyson product respectively \cite{Ekhad-zeil,Gessel-Lv-Xin-Zhou2008,Sills,KLW,XZ2023}.
Besides the constant term of the $q$-Dyson product, Sills \cite{Sills} conjectured three other coefficients.
Lv, Xin and Zhou \cite{LXZ} proved Sills' conjectures and obtained all the first-layer coefficients of the $q$-Dyson product.
That is
\begin{equation}\label{e-firstlayer}
 \CT_{\x} \frac{x_{j_1}^{p_1}\cdots x_{j_t}^{p_t}}{x_{i_1}\cdots x_{i_s}}\prod_{1\leq i<j\leq n}(x_i/x_j)_{a_i}(qx_j/x_i)_{a_j},
\end{equation}
where all the $i_1,\dots,i_s,j_1,\dots,j_t$ are distinct and the $p_i$ are positive integers.
Here we omit the explicit formula due to its complexity.
Sills' ex-conjectures correspond to the cases $s,t\leq 2$ in \eqref{e-firstlayer}.
The equal parameter case, i.e., $a_1=\cdots=a_n=c$ of \eqref{e-firstlayer} was obtained by Stembridge \cite{stembridge1987} in his first-layer formulas for the characters of $\mathrm{SL}(n,\mathbb{C})$.

In 2000, Kadell made a remarkable conjecture
on a symmetric function generalization of the $q$-Dyson identity.
Let $X=(x_1,x_2,\dots)$ be an alphabet of countably many variables.
Define the $r$th complete symmetric function $h_r(X)$
in terms of its generating function as
\begin{equation}\label{e-gfcomplete}
\sum_{r\geq 0} z^r h_r(X)=\prod_{i\geq 1}
\frac{1}{1-zx_i}.
\end{equation}
More generally, the complete symmetric function indexed by a partition $\lambda=(\lambda_1,\lambda_2,\dots)$ is defined as
\[
h_{\lambda}:=h_{\lambda_1}h_{\lambda_2}\cdots,
\]
where a partition $\lambda$ is a weakly decreasing sequence of nonnegative integers such that only finitely many $\lambda_i$ are positive.
For $a:=(a_1,\dots,a_n)$ a sequence of nonnegative integers, let $\x^{(a)}$ denote
the alphabet
\begin{equation}\label{alphabet-x}
\x^{(a)}=(x_1,x_1q,\dots,x_1q^{a_1-1},\dots,x_n,x_nq,\dots,x_nq^{a_n-1})
\end{equation}
of cardinality $|a|:=a_1+\cdots+a_n$, and define
the generalized $q$-Dyson constant term
\begin{equation}\label{D-Kadell}
D_{v,\lambda}(a)=\CT_{\x}
x^{-v}h_{\lambda}\big(\x^{(a)}\big)
\prod_{1\leq i<j\leq n}
(x_i/x_j)_{a_i}(qx_j/x_i)_{a_j}.
\end{equation}
Here $v=(v_1,\dots,v_n)\in\mathbb{Z}^n$,
$x^v$ denotes the monomial $x_1^{v_1}\cdots x_n^{v_n}$
and $\lambda$ is a partition.
For the constant term \eqref{D-Kadell},
Kadell formulated the
following conjecture \cite[Conjecture 4]{Kadell2000}.
\begin{thm}[Kadell's ex-conjecture]\label{conj-Kadell}
For $r$ a positive integer and $v$ a vector of nonnegative integers such that $|v|=r$,
\begin{equation}\label{kadellconj}
D_{v,(r)}(a)=
\begin{cases}
\displaystyle
\frac{q^{\sum_{i=k+1}^n a_i}(1-q^{a_k})(q^{|a|})_r}{(1-q^{|a|})(q^{|a|-a_k+1})_r}
\prod_{i=1}^n\qbinom{a_i+\cdots+a_n}{a_i}
& \text{if $v=(0^{k-1},r,0^{n-k})$}, \\[6mm]
0 & \text{otherwise},
\end{cases}
\end{equation}
where $\qbinom{n}{c}=(q^{n-c+1})_c/(q)_c$ is the $q$-binomial coefficient for nonnegative integers $n$ and $c$.
\end{thm}
In fact Kadell only considered $v=(r,0^{n-1})$ in his conjecture, but the
more general statement given above is what was proved by K\'{a}rolyi,
Lascoux and Warnaar in \cite[Theorem 1.3]{KLW} using multivariable Lagrange
interpolation and key polynomials.
For a sequence $u=(u_0,\dots,u_n)$ of integers, denote
by $u^{+}$ the sequence obtained from $u$ by ordering the $u_i$
in weakly decreasing order. K\'{a}rolyi et al.\ also
proved a closed-form expression for
$D_{v,v^{+}}(a)$ in the case when all the parts of $v$ are nonnegative integers and have multiplicity one, i.e., all the $v_i\geq 0$ and $v_i\neq v_j$ for all $1\leq i<j\leq n$.
Subsequently, Cai \cite{cai} gave an inductive proof of Kadell's conjecture.
Recently, we obtained some further results on $D_{v,v^{+}}(a)$, see \cite{Zhou, Zhou-JCTA}.

In 1982, Macdonald discovered that the equal parameter case of the $q$-Dyson identity, i.e., the $a_1=\cdots=a_n=c$ case, can be formulated as a combinatorial
identity for the root system $\mathrm{A}_{n-1}$.
Then he conjectured a constant term identity for arbitrary root systems \cite{macdonald82}:
\begin{thm}[Macdonald's ex-conjecture]
For a nonnegative integer $c$,
\[
\CT \,\prod_{\alpha\in R^+}(e^{-\alpha})_c(qe^{\alpha})_c=
\prod_{i=1}^r\qbinom{d_ic}{c}.
\]
Here $R$ is a reduced irreducible finite root system of rank $r$, $R^+$ is the set of positive roots and $d_1,\dots,d_r$ are the degrees of the fundamental invariants.
\end{thm}
Initially, many cases of Macdonald's conjecture were proved on a case by case basis \cite{Askey80,GG,Hab,kadell1,zeil-87}.
A uniform proof for $q=1$ was first found by Opdam \cite{opdam} using hypergeometric shift operators.
Thirteen years after Macdonald presented his constant term conjecture,
Cherednik \cite{cherednik} finally gave a case-free proof based on his double affine Hecke algebra. For more on the extensive literature of Macdonald's constant term conjecture we refer the reader to \cite{FW} and  references therein.

Another important generalization of the equal parameter case of the $q$-Dyson identity
is the $q$-Morris constant term identity.
\begin{thm}[Morris' ex-conjecture]
For nonnegative integers $a,b,c$,
\begin{equation}\label{q-Morris}
\CT_{\x} \prod_{i=1}^{n}\Big(\frac{x_{0}}{x_{i}}\Big)_a
\Big(\frac{qx_{i}}{x_{0}}\Big)_b\prod_{1\leq i<j\leq n}
\Big(\frac{x_{i}}{x_{j}}\Big)_c\Big(\frac{qx_{j}}{x_{i}}\Big)_c
=\prod_{i=0}^{n-1}\frac{(q)_{a+b+ic}(q)_{(i+1)c}}{(q)_{a+ic}(q)_{b+ic}(q)_{c}}.
\end{equation}
\end{thm}
Since the Laurent polynomial in \eqref{q-Morris} is homogeneous in $x_0,\dots,x_n$, we can take, for example $x_0=1$, without changing
the constant term.
The $q=1$ case of \eqref{q-Morris} was proved by Morris \cite{Morris1982} in his Ph.D thesis.
He also showed that the $q=1$ case is equivalent to the Selberg integral \cite{Selberg} and conjectured the above $q$-analogue.
It is known that the $q$-Morris identity \eqref{q-Morris} is equivalent to Askey's $q$-Selberg
ex-conjecture. We state the result below.

For $0<q<1$, define the $q$-integral or the Jackson integral as
\[
\int_{0}^1f(t)\mathrm{d}_q t:=(1-q)\sum_{k=0}^{\infty}f(q^k)q^k,
\]
where $f$ is a function for which the right hand side converges.
In 1980, Askey \cite{Askey} formulated a $q$-analogous conjecture of the Selberg integral \cite{Selberg}.
\begin{thm}[Askey's ex-conjecture]
Let $k$ be a positive integer and $\alpha, \beta$ be complex parameters such that
$\mathrm{Re}(\alpha)>0$, $\mathrm{Re}(\beta)>0$.
Then
\begin{multline}\label{e-Askey}
\int\limits_{[0,1]^n}\prod_{i=1}^nx_i^{\alpha-1}(qx_i)_{\beta-1}\prod_{1\leq i<j\leq n}x_j^{2k}
(q^{1-k}x_i/x_j)_{2k}\mathrm{d}_qx_{1}\cdots \mathrm{d}_qx_{n}\\
=q^{\alpha k\binom{n}{2}+2k^2\binom{n}{3}}\prod_{j=0}^{n-1}
\frac{\Gamma_q\big(\alpha+k j\big)\Gamma_q\big(\beta+k j\big)\Gamma_q\big(1+k (j+1)\big)}
{\Gamma_q\big(\alpha+\beta+k(n+j-1)\big)\Gamma_q(1+k)},
\end{multline}
where $\Gamma_q(z)=(q)_{z-1}(1-q)^{1-z}$ is the $q$-gamma function
for $z\in \mathbb{C}\setminus\{0,-1,-2,\dots\}$.
\end{thm}
The $q=1$ case of \eqref{e-Askey} is the famous Selberg integral \cite{Selberg}.
When $n=1$, the Selberg integral reduces to the Euler beta integral.
In 1988, Askey's conjecture was proved independently by Habsieger \cite{Habsieger} and Kadell \cite{Kad}. Both Habsieger and Kadell turned the integral \eqref{e-Askey} into an equivalent constant term identity \eqref{q-Morris}. Hence, the identity \eqref{q-Morris}
is usually referred to as the Habsieger--Kadell $q$-Morris identity,
while its equivalent integral \eqref{e-Askey} is called the Askey--Habsieger--Kadell integral.
For the importance of the Selberg integral, see a review \cite{FW}.

Now we introduce Baker and Forrester's conjecture on a generalization of the $q$-Morris identity.
Let $n_0,n_1,\dots,n_p$ be positive integers such that $n_0+\cdots+n_p=n$. Set
$N_0=\{1,\dots,n_0\}$, $N_1=\{n_0+1,\dots,n_0+n_1\},\dots,
N_p=\{n_0+\cdots+n_{p-1}+1,\dots,n\}$.
Denote
\begin{equation}\label{def-Tij}
\varepsilon_{ij}=\begin{cases}
1 \quad \text{if} \ \{i,j\}\subseteq N_l \ \text{for some}\ l\in \{1,\dots,p\},\\
0 \quad \text{otherwise}.\end{cases}
\end{equation}
Define
\begin{equation}\label{L}
L_{n_0,\dots,n_p}(a,b,c;x)=
\prod_{i=1}^{n}(x_0/x_i)_a(qx_i/x_0)_b
\prod_{1\leq i<j\leq n}(x_i/x_j)_{c+\varepsilon_{ij}}(qx_j/x_i)_{c+\varepsilon_{ij}},
\end{equation}
and
\[
D_n\big((n_0,\dots,n_p);a,b,c\big)
=\CT_{\x} L_{n_0,\dots,n_p}(a,b,c;x).
\]
In 1998, Baker and Forrester \cite[Conjecture 2.2]{BF} conjectured a recursion
for $D_n\big((n_0,\dots,n_p);a,b,c\big)$. We will refer to this in detail in Theorem~\ref{thm-1}.

By the definition of the $\varepsilon_{ij}$ in \eqref{def-Tij}, if $p=0$ then all the $\varepsilon_{ij}=0$.
It follows that in the $p=0$ case $D_{n}(n_0;a,b,c)$ is just the constant term in the $q$-Morris identity \eqref{q-Morris}.

There have been several attempts to prove the Baker--Forrester conjecture in the $p=1$ case
\cite{Gessel-Lv-Xin-Zhou2008,Kaneko02,Kaneko03,Hamada}.
Following their discovery of the ``Proof from the book'' of the $q$-Dyson theorem using the combinatorial Nullstellensatz \cite{KN},
K\'{a}rolyi, Nagy, Petrov and Volkov \cite{KNPV} proved this case in 2015. We state the result in the next theorem.
\begin{thm}\label{thm-qForrester}
For $n_0+n_1=n$,
\begin{equation}\label{main-second}
D_n\big((n_0,n_1);a,b,c\big)
=\prod_{j=2}^{n-n_0}(1-q^{j(c+1)})
\prod_{j=0}^{n-1}\frac{(q^{a+jc+\chi
(j>n_0)(j-n_0)+1})_{b}
(q)_{(j+1)c+\chi(j>n_0)(j-n_0)}}
{(q)_{b+jc+\chi(j>n_0)(j-n_0)}(q)_{c+\chi(j>n_0)}},
\end{equation}
where $\chi(S)$ is the truth function which is 1 if $S$ is true and 0 otherwise.
\end{thm}
K\'{a}rolyi et al. \cite{KNPV} actually gave an Aomoto-type \cite{Aomoto} extension of the above theorem by adding
an extra parameter.
Recently, we \cite{XZ2023} obtained several symmetric function generalizations of
Theorem~\ref{thm-qForrester} and its Aomoto-type extension.

In the general $p+1$ case, Baratta \cite[Theorem 12]{Baratta} proved the $a=b=0$ case.
In this paper, we prove the $q$-Baker--Forrester conjecture for all $p\geq 1$.
\begin{thm}\label{thm-1}
Assume $n_k=\max\{n_1,\dots,n_p\}$. Then
\begin{multline}\label{e-recursion}
D_n\big((n_0,\dots,n_p);a,b,c\big)
=\frac{(1-q^{n_k(c+1)})(q^{a+(n-1)c+n_k})_b}{(1-q^{c+1})(q^{(n-1)c+n_k})_b}
\qbinom{nc+n_k-1}{c}\\
\times D_{n-1}\big((n_0,\dots,n_{k-1},n_{k}-1,n_{k+1}\dots,n_p);a,b,c\big).
\end{multline}
\end{thm}
We refer to $D_n\big((n_0,\dots,n_p);a,b,c\big)$ as $D_n(a)$ for short throughout this paper if there is no risk of ambiguity.
Note that the value of $k$ in the theorem may not be unique, say
$n_{k_1}=\cdots=n_{k_s}=\max\{n_1,\dots,n_p\}$. Then \eqref{e-recursion} holds for any
such $n_{k_i}$. The original restriction on the $n_i$ in \cite{BF} is
$n_p>n_j$ for all $j=1,\dots,p-1$.
However, the assumption stated in the theorem is sufficient to derive a closed-form formula for $D_n(a)$. Thus, we propose a slight variation.
Also note that in the definition of $n_k=\max\{n_1,\dots,n_p\}$, $n_0$ is not included in the set $\{n_1,\dots,n_p\}$. As mentioned above, the Morris-type constant term identities are equivalent
to the Selberg-type integrals. Therefore, the recursion given in \eqref{e-recursion} suggests
a corresponding integral recursion. However, the integral recursion is too complicated to be included here.

This paper is organized as follows. In Section~\ref{s-sketch}, we present a proof sketch of our main result, Theorem~\ref{thm-1}.
In Section~\ref{sec-split}, we derive a splitting formula. Using this formula, we get a recursion for $D_n(0)$ and characterize a family of
vanishing coefficients of $\prod_{1\leq i<j\leq n}(x_i/x_j)_{c+\varepsilon_{ij}}(qx_j/x_i)_{c+\varepsilon_{ij}}$.
In Section~\ref{sec-GX}, we introduce the primary process of the Gessel--Xin method for evaluating the constant terms.
In Section~\ref{sec-roots}, we determine all the roots of $D_n(a)$. In the last section, we complete the proof of Lemma~\ref{lem-Q}
in Section~\ref{sec-roots}.

\section{A sketch of the proof of Theorem~\ref{thm-1}}\label{s-sketch}

In this section, we provide a sketch of the proof of Theorem~\ref{thm-1}.
In Subsection~\ref{subsec-mainsteps}, we present the main steps of our proof of Theorem~\ref{thm-1}.
In Subsection~\ref{subsec-poly}, we show that $D_n(a)$ is a polynomial in $q^a$.
In Subsection~\ref{subsec-roots}, we describe the roots of $D_n(a)$.
In Subsection~\ref{subsec-rati}, we demonstrate that it is sufficient to derive an expression for $D_n(a)$ for sufficiently large $c$.
In Subsection~\ref{subsec-recursion}, we give a recursion for $D_n(0)$.
In Subsection~\ref{subsec-proof}, we prove Theorem~\ref{thm-1}.

\subsection{The main steps}\label{subsec-mainsteps}

The basic idea for proving Theorem~\ref{thm-1} is the following.
A polynomial of degree at most $d$ is determined by $d+1$ distinct values.
We obtain \eqref{e-recursion} through the next four steps.
\begin{enumerate}
\item \textbf{Polynomiality:} View $D_n(a)$ as a polynomial in $q^a$ of degree at most $nb$;
\item \textbf{Roots:} Explicitly determine all the roots of $D_n(a)$
under the assumption that $c$ is sufficiently large;
\item \textbf{A non-vanishing value:} Obtain a recursion for $D_n(0)$ (i.e., the $q^a=1$ case),
and use this to obtain a recursion for $D_n(a)$;
\item \textbf{Rationality:} Extend the recursion for $D_n(a)$ to all nonnegative integers
$c$ by a rationality argument.
\end{enumerate}

Below we explain the above steps in further detail.

The first step is routine, see Corollary~\ref{cor-poly} below.
To view a constant term as a polynomial in one parameter originates from
Gessel and Xin's Laurent series proof \cite{GX} of
the Zeilberger--Bressoud $q$-Dyson identity \eqref{q-Dyson}.
We utilized this idea for obtaining a series of constant term identities in \cite{XZ,XZ2023,LXZ,Zhou,Zhou23}.

The second step is quite involved.
We find that $D_n(a)$ can vanish only when $a$ is a negative integer.
Hence, we need to extend the original definition of $D_n(a)$ to all the integers $a$ using the theory of
iterated Laurent series \cite{xinresidue}. For a negative integer $a$, we actually
need to take constant term of a rational function, rather than a Laurent polynomial.
The original Gessel--Xin method can only
deal with the constant term of a rational function, say $f(x_i)$, with negative degrees.
Here we say that the degree of the rational function $f(x_i)$ is the degree in $x_i$
of its numerator minus the degree in $x_i$ of its denominator. We extended their method to deal with
rational functions with non-positive degrees to obtain the first-layer coefficients
of the $q$-Dyson product \cite{LXZ}. We mentioned this in \eqref{e-firstlayer} in the introduction.
For $D_n(a)$, not all the roots can be determined using the
original Gessel--Xin method. This is caused by the non-symmetry of the $x_i$'s.
For some roots, we have to tackle
the problem of how to obtain the constant term of a rational function with positive degree.
This has been a long-standing hurdle \cite{Gessel-Lv-Xin-Zhou2008}.
Recently, we discovered that rational functions of the $q$-Baker--Forrester type with 2 components
(i.e. the $p=1$ case of $D_n(a)$) can be transformed into Laurent polynomials under certain conditions \cite{XZ2023}.
In this paper, build on this result, we deal with the general case of $p+1$ components.
Then, using Cai's splitting idea \cite{cai} for proving constant term identities,
we find some new splitting formulas for the $q$-Baker--Forrester type Laurent polynomials.
Using these new splitting formulas, we manage to deal with all the
zeros of $D_n(a)$.

To explicitly determine the closed-form formula for $D_n(a)$ as a polynomial,
besides all its roots, we require one extra non-vanishing value.
To achieve this, we find a recursion for $D_n(0)$ using the above mentioned splitting formulas.
Together with all the roots of $D_n(a)$, this completes the recursion \eqref{e-recursion}
in Theorem~\ref{thm-1} for sufficiently large $c$.

The assumption that $c$ is sufficiently large is imposed to avoid
multiple roots for $D_n(a)$ in Step (2).
In Step (4), a rationality result, see Proposition \ref{prop-rationality}, is used to
extend the recursion \eqref{e-recursion} to all nonnegative integers $c$.

\subsection{The polynomiality}\label{subsec-poly}
The next polynomiality result is a special case of \cite[Lemma 2.2]{XZ}
and \cite[Lemma 5.1]{XZ2023}.
\begin{lem}\label{lem1}
Let $A(x_1,\dots,x_n)$ be an arbitrary Laurent polynomial independent of $a$ and $x_0$.
Then, for fixed nonnegative integers $b$,
\begin{equation}\label{p1}
\CT_{\x} \prod_{i=1}^n (x_0/x_i)_a(qx_i/x_0)_b A(x_1,\dots,x_n)
\end{equation}
is a polynomial in $q^a$ of degree at most $nb$.
\end{lem}
Then, by taking
\[
A(x_1,\dots,x_n)=\prod_{1\leq i<j\leq n}(x_i/x_j)_{c+\varepsilon_{ij}}(qx_j/x_i)_{c+\varepsilon_{ij}}
\]
in Lemma~\ref{lem1}, we complete Step (1) in Subsection~\ref{subsec-mainsteps}. We state the result in the next corollary.
\begin{cor}\label{cor-poly}
Assume all the parameters of $D_n(a)$, except for $a$, are fixed. The constant term $D_n(a)$ is a polynomial in $q^a$
with degree at most $nb$.
\end{cor}

\subsection{The roots}\label{subsec-roots}
To describe the roots of $D_n(a)$, we need to sort $n_1,\dots,n_p$.
Assume $n_{s_1},\dots,n_{s_p}$ is a rearrangement of $n_1,\dots,n_p$
such that $1\leq n_{s_1}\leq \cdots \leq n_{s_p}$.
Note that if one $n_i=0$ then it is in fact the $p-1$ case.
Therefore, we can assume all the $n_i$ to be positive integers.
Also note that $n_0$ is not sorted and it is not included
in the $n_{s_j}$.
To explicitly describe all the roots of $D_n(a)$, we begin with the patterns
from $p=0,1,2$.

For the $p=0$ case, the constant term $D_n(n_0;a,b,c)$ vanishes when
\begin{equation}\label{e-root0}
-a\in\{ic+1,ic+2,\dots,ic+b\mid i=0,\dots,n-1\}.
\end{equation}

For the $p=1$ case, the constant term $D_n\big((n_0,n_1);a,b,c\big)$ vanishes when
\begin{multline}\label{e-root1}
-a\in\{ic+1,ic+2,\dots,ic+b\mid i=0,\dots,n_0\} \\
\cup \{ic+i-n_0+1,ic+i-n_0+2,\dots,ic+i-n_0+b\mid i=n_0+1,\dots,n-1\}.
\end{multline}
We can represent the sets in \eqref{e-root1} as
\begin{subequations}\label{d-roots1}
\begin{equation}\label{d-roots1-1}
\left.
\begin{array}{cccccccccc}
&1 &\dots &b \\
&c+1 &\dots &c+b\\
&\vdots  &\vdots &\vdots \\
&n_0c+1 &\dots &n_0c+b
\end{array}\right\} \text{$n_0+1$ rows}
\end{equation}
\begin{equation}\label{d-roots1-2}
\left.
\begin{array}{cccccccccc}
&(n_0+1)c+2 &\dots &(n_0+1)c+b+1\\
&(n_0+2)c+3 &\dots &(n_0+2)c+b+2\\
&\vdots  &\vdots &\vdots \\
&(n-1)c+n_1 &\dots &(n-1)c+b+n_1-1
\end{array}\right\} \text{$n_1-1$ rows}.
\end{equation}
\end{subequations}
Each row in \eqref{d-roots1} contains $b$ elements, and there are $n$ rows in total.
The first $n_0+1$ or $n-(n_1-1)$ rows in \eqref{d-roots1-1}
keep the same pattern as that in \eqref{e-root0}.

For the $p=2$ case, assume $n_1$ and $n_2$ are rearranged increasingly as $n_{s_1}, n_{s_2}$.
Then the constant term $D_n\big((n_0,n_1,n_2);a,b,c\big)$ vanishes when
\[
-a\in \bigcup_{i=0}^{n_0+1}R^0_i\bigcup_{i=n_0+2}^{n_0+2n_{s_1}-1} R^1_i\bigcup_{i=n_0+2n_{s_1}}^{n-1} R^2_i.
\]
Here
\[
R^0_i=\{ic+1,ic+2,\dots,ic+b\}
\]
for $i=0,\dots,n_0+1$,
\begin{equation*}
R^1_i=\Big\{ic+\Big\lfloor\frac{i-n_0}{2}\Big\rfloor+1,ic+\Big\lfloor\frac{i-n_0}{2}\Big\rfloor+2,\dots,
ic+\Big\lfloor\frac{i-n_0}{2}\Big\rfloor+b\Big\}
\end{equation*}
for $i=n_0+2,\dots,n_0+2n_{s_1}-1$,
and
\begin{equation*}
R^2_i=\{ic+i-n_0-n_{s_1}+1,ic+i-n_0-n_{s_1}+2,\dots,
ic+i-n_0-n_{s_1}+b\}
\end{equation*}
for $i=n_0+2n_{s_1},\dots,n-1$.
We represent the elements in $R^0_i, R^1_i, R^2_i$ in rows as follows respectively:
\begin{subequations}\label{d-roots2}
\begin{equation}\label{d-roots2-1}
\left.
\begin{array}{cccccccccc}
&1 &\dots &b \\
&c+1 &\dots &c+b\\
&\vdots  &\vdots &\vdots \\
&(n_0+1)c+1 &\dots &(n_0+1)c+b
\end{array}\right\} \text{$n_0+2$ rows}
\end{equation}
\begin{equation}\label{d-roots2-2}
\left.
\begin{array}{cccccccccc}
&(n-n_1-n_2+2)c+2 &\dots &(n-n_1-n_2+2)c+b+1\\
&(n-n_1-n_2+3)c+2 &\dots &(n-n_1-n_2+3)c+b+1\\
&(n-n_1-n_2+4)c+3 &\dots &(n-n_1-n_2+4)c+b+2\\
&(n-n_1-n_2+5)c+3 &\dots &(n-n_1-n_2+5)c+b+2\\
&\vdots  &\vdots &\vdots \\
&(n+n_{s_1}-n_{s_2}-2)c+n_{s_1} &\dots &(n+n_{s_1}-n_{s_2}-2)c+n_{s_1}+b-1\\
&(n+n_{s_1}-n_{s_2}-1)c+n_{s_1} &\dots &(n+n_{s_1}-n_{s_2}-1)c+n_{s_1}+b-1
\end{array}\right\} \text{$2(n_{s_1}-1)$ rows}
\end{equation}
\begin{equation}\label{d-roots2-3}
\left.
\begin{array}{cccccccccc}
&(n+n_{s_1}-n_{s_2})c+n_{s_1}+1 &\dots &(n+n_{s_1}-n_{s_2})c+n_{s_1}+b\\
&(n+n_{s_1}-n_{s_2}+1)c+n_{s_1}+2 &\dots &(n+n_{s_1}-n_{s_2}+1)c+n_{s_1}+b+1\\
&\vdots  &\vdots &\vdots \\
&(n-1)c+n_{s_2} &\dots &(n-1)c+n_{s_2}+b-1
\end{array}\right\} \text{$n_{s_2}-n_{s_1}$ rows.}
\end{equation}
\end{subequations}
Note that in \eqref{d-roots2} each row contains $b$ elements,
the elements in each row are consecutive integers arranged in an increasing order.
There are $n$ rows in total.
The first $n_0+2$ or $n-(n_1-1)-(n_2-1)$ rows in \eqref{d-roots2-1}
are of the same pattern as that in \eqref{e-root0} and \eqref{d-roots1-1}.
Take the rows of \eqref{d-roots2-2} every other row from the first line
(represent these rows using their first elements):
\[
(n-n_1-n_2+2)c+2,(n-n_1-n_2+4)c+3,\dots,(n+n_{s_1}-n_{s_2}-2)c+n_{s_1}.
\]
We can see that these $n_{s_1}-1$ rows are of the similar pattern as that
in \eqref{d-roots1-2}, except the coefficients of $c$ differ by 2.
The first elements of the remaining rows of \eqref{d-roots2-2} and \eqref{d-roots2-3}
are
\begin{multline*}
(n-n_1-n_2+3)c+2,(n-n_1-n_2+5)c+3,\dots,(n+n_{s_1}-n_{s_2}-1)c+n_{s_1},\\
(n+n_{s_1}-n_{s_2})c+n_{s_1}+1,(n+n_{s_1}-n_{s_2}+1)c+n_{s_1}+2,
\dots,(n-1)c+n_{s_2}.
\end{multline*}
We can see that these $n_{s_2}-1$ rows are of the similar pattern as that
in \eqref{d-roots1-2}, except the coefficients of $c$ differ by 2 or 1.
We can arrange the elements in \eqref{d-roots2} in this way:
\begin{enumerate}
\item Arrange the consecutive $b$ integers increasingly in each row.
Thus, the elements in every row are determined by the first element of that row;
\item
The roots in \eqref{d-roots2-1} are routine. There are $n-(n_1-1)-(n_2-1)=n_0+2$ rows,
and the first element of each row is of the form $ic+1$ for $i=0,\dots,n_0+1$;
\item
The numbers $n_1$ and $n_2$ determine $n_1-1$ and $n_2-1$ rows respectively.
We can think these rows are arranged alternatively, as shown below.
\begin{equation}\label{root-p2}
\xymatrix@R=0.5ex{
&(n-n_1-n_2+2)c+2 \quad \dots \quad (n-n_1-n_2+2)c+b+1&&\\
&(n-n_1-n_2+3)c+2 \quad \dots \quad (n-n_1-n_2+3)c+b+1& &\\
&(n-n_1-n_2+4)c+3 \quad \dots \quad (n-n_1-n_2+4)c+b+2&&\text{$n_{s_1}-1$ rows}\\
&(n-n_1-n_2+5)c+3 \quad \dots \quad (n-n_1-n_2+5)c+b+2&&\\
&\vdots  \quad \quad \quad  \quad \quad \quad \vdots  \quad \quad \quad \quad \quad \quad \vdots &&\\
&(n+n_{s_1}-n_{s_2}-2)c+n_{s_1}\quad \dots \quad(n+n_{s_1}-n_{s_2}-2)c+n_{s_1}+b-1&&\\
&(n+n_{s_1}-n_{s_2}-1)c+n_{s_1}\quad \dots \quad(n+n_{s_1}-n_{s_2}-1)c+n_{s_1}+b-1&&\text{$n_{s_2}-1$ rows.}\\
&&&\\
&&&\\
&&&\\
&(n+n_{s_1}-n_{s_2})c+n_{s_1}+1 \quad \dots \quad (n+n_{s_1}-n_{s_2})c+n_{s_1}+b&&\\
&(n+n_{s_1}-n_{s_2}+1)c+n_{s_1}+2\quad \dots \quad (n+n_{s_1}-n_{s_2}+1)c+n_{s_1}+b+1&&\\
&\vdots  \quad \quad \quad  \quad \quad \quad \vdots  \quad \quad \quad \quad \quad \quad \vdots &&\\
&(n-1)c+n_{s_2} \quad \dots \quad (n-1)c+n_{s_2}+n_{s_2}+b-1&&
\ar"1,2";"3,4" \ar"3,2";"3,4" \ar"6,2";"3,4"
\ar"2,2";"7,4" \ar"4,2";"7,4" \ar"7,2";"7,4" \ar"11,2";"7,4" \ar"12,2";"7,4" \ar"14,2";"7,4"\ar@{-}"9,1";"9,3"
}
\end{equation}
\end{enumerate}

The reader may conclude the general pattern for any nonnegative integer $p$. Notice that in this general case
$n_0=n-\sum_{i=1}^pn_{i}$, and $1\leq n_{s_1}\leq \dots\leq n_{s_p}$ is a rearrangement of $n_1,\dots,n_p$.
In this case, the constant term $D_n\big((n_0,\dots,n_p);a,b,c\big)$ vanishes
when $-a$ belongs to the following sets:
\begin{equation}\label{def-R0}
R^0_i=\{ic+1,ic+2,\dots,ic+b\}
\end{equation}
for $i=0,\dots,n_0+p-1$;
\[
R^1_i=\Big\{ic+\Big\lfloor \frac{i-n_0}{p} \Big\rfloor+1, \dots,
ic+\Big\lfloor \frac{i-n_0}{p} \Big\rfloor+b
\Big\}
\]
for $i=n-\sum_{i=1}^pn_{i}+p,\dots,n-\sum_{\substack{i=1\\i\neq s_1}}^pn_{i}+(p-1)n_{s_1}-1$;
\[
R^2_i=\Big\{ic+\Big\lfloor \frac{i-n_0-n_{s_1}}{p-1} \Big\rfloor+1, \dots,
ic+\Big\lfloor \frac{i-n_0-n_{s_1}}{p-1} \Big\rfloor+b
\Big\}
\]
for $i=n-\sum_{\substack{i=1\\i\neq s_1}}^pn_{i}+(p-1)n_{s_1},
\dots,n-\sum_{\substack{i=1\\i\neq s_1,s_2}}^pn_{i}+(p-2)n_{s_2}-1$;
\[
R^3_i=\Big\{ic+\Big\lfloor \frac{i-n_0-n_{s_1}-n_{s_2}}{p-2} \Big\rfloor+1, \dots,
ic+\Big\lfloor \frac{i-n_0-n_{s_1}-n_{s_2}}{p-2} \Big\rfloor+b
\Big\}
\]
for $i=n-\sum_{\substack{i=1\\i\neq s_1,s_2}}^pn_{i}+(p-2)n_{s_2},
\dots,n-\sum_{\substack{i=1\\i\neq s_1,s_2,s_3}}^pn_{i}+(p-3)n_{s_3}-1$;
and so on. Finally,
\begin{equation}\label{R_p}
R^p_i=\Big\{ic+i-n_0-n_{s_1}-\cdots-n_{s_{p-1}}+1, \dots,
ic+i-n_0-n_{s_1}-\cdots-n_{s_{p-1}}+b
\Big\}
\end{equation}
for $i=n-\sum^p_{\substack{i=1\\i\neq s_1,\dots,s_{p-1}}}n_i+n_{s_{p-1}}, \dots,n-1$.
In general, for a fixed $j\in \{1,\dots,p\}$ we can define
\begin{equation}\label{def-Rj}
R^j_i:=\Big\{ic+\Big\lfloor \frac{i-n_0-\sum_{k=1}^{j-1}n_{s_k}}{p-j+1} \Big\rfloor+1, \dots,
ic+\Big\lfloor \frac{i-n_0-\sum_{k=1}^{j-1}n_{s_k}}{p-j+1} \Big\rfloor+b
\Big\}
\end{equation}
for $i=n-\sum_{i=j}^pn_{s_i}+(p-j+1)n_{s_{j-1}},\dots,
n-\sum_{i=j+1}^pn_{s_i}+(p-j)n_{s_j}-1$.
Here $n_{s_0}:=1$ and a summation is defined to be zero as usual all over this paper if the lower bound exceeds the upper bound.
The sets $R^0_i$ should be defined as in \eqref{def-R0} separately.
Note that in $R_i^j$ the index $j$ is determined by $i$. We use this redundant index $j$ to categorize the sets $R_i^j$ according to the number $n_{s_j}$.
We can also arrange all the elements of $R^j_i$
as $nb$ rows similar to that in \eqref{d-roots2}. The first $n_0+p$ rows in $R^0_i$ are routine as in
\eqref{d-roots1-1} and \eqref{d-roots2-1}.
The numbers $n_j$ determine $n_j-1$ rows for $j=1,\dots,p$ respectively
similar to the $p=2$ case in \eqref{root-p2}.
Define $R$ to be the union of all the sets $R^j_i$. That is
\begin{equation}\label{def-roots}
R=\bigcup_{i=0}^{n-1}R^j_i,
\end{equation}
where $j\in \{0,\dots,p\}$ is determined by the corresponding $i$ in $R_i^j$ as mentioned above.
Then we can state our main lemma.
\begin{lem}\label{lem-roots}
The constant term $D_n\big((n_0,\dots,n_p);a,b,c\big)$
vanishes for $-a\in R$.
\end{lem}
The proof of this lemma is quite involved, we leave this to Section~\ref{sec-roots}.

\subsection{A rationality result}\label{subsec-rati}
The number of elements $|R|$ in \eqref{def-roots} may not be $nb$
if $D_n(a)$ has multiple roots.
It is easy to find this in \eqref{d-roots1} if $c<b$ in the $p=1$ case.
If $c<b$ then
the adjacent rows in \eqref{d-roots1} can have the same element.
If $c$ is sufficiently large, then the union of the $R^j_i$ in \eqref{def-roots} is a disjoint union, and $|R|=nb$ is the same as the upper bound of the degree of $D_n(a)$. In this case, we only need to find one extra non-vanishing value for $D_n(a)$
to determine its closed-form expression.
It is evident that the explicit lower bound for the variable $c$ to ensure $|R|=nb$ for any value of $p$ is $c\geq b$.
We use the next result to extend the expression for $D_n(a)$ from sufficiently large $c$
to all nonnegative integers $c$.
\begin{prop}\cite[Corollary 7.4]{Zhou23}\label{prop-rationality}
Let $d$ be a nonnegative integer independent of $c$ and $H$ be a homogeneous Laurent polynomial in $x_0,x_1,\dots,x_n$ of degree 0. If an expression for
\begin{equation}\label{e-H}
\CT_{\x} H\prod_{1\leq i<j\leq n}
\Big(\frac{x_{i}}{x_{j}}\Big)_{c}\Big(q\frac{x_{j}}{x_{i}}\Big)_{c}
\end{equation}
holds for $c\geq d$, then it also holds for all nonnegative integers $c$.
\end{prop}
By taking
\[
H=\prod_{i=1}^{n}(x_0/x_i)_a(qx_i/x_0)_b
\prod_{1\leq i<j\leq n}(q^cx_i/x_j)_{\varepsilon_{ij}}(q^{c+1}x_j/x_i)_{\varepsilon_{ij}}
\]
in Proposition~\ref{prop-rationality},
the constant term in \eqref{e-H} is just $D_n(a)$.
Thus, to determine $D_n(a)$, it is sufficient to find its closed-form expression for $c\geq b$. In the remainder of the paper, we always assume that $c\geq b$.

\subsection{A recursion for $D_n(0)$}\label{subsec-recursion}
Our last step to determine $D_n(a)$ is to obtain $D_n(0)$.
It is easy to see that
\begin{equation}\label{e-0}
D_n(0)=D_n\big((n_0,\dots,n_p);0,b,c\big)=D_n\big((n_0,\dots,n_p);0,0,c\big)
\end{equation}
by homogeneity.
We obtain a recursion
for $D_n(0)$ in the next lemma.
\begin{lem}\label{lem-0}
Assume $n_k=\max\{n_1,\dots,n_p\}$. Then
\begin{equation}\label{e-rec0}
D_n\big((n_0,\dots,n_p);0,0,c\big)
=\frac{1-q^{n_k(c+1)}}{1-q^{c+1}}\qbinom{nc+n_k-1}{c}
D_{n-1}\big((n_0,\dots,n_{k-1},n_{k}-1,n_{k+1},\dots,n_p);0,0,c\big).
\end{equation}
\end{lem}
We will give a proof of Lemma~\ref{lem-0} in Section~\ref{sec-split}.
It is not hard to see that one can obtain a closed-form expression for $D_n(0)$ using the recursion \eqref{e-rec0} and the equal parameter case of the $q$-Dyson identity \eqref{q-Dyson}.
However, it is too complicated to present its explicit expression here.

\subsection{A proof of Theorem~\ref{thm-1}}\label{subsec-proof}
By Corollary~\ref{cor-poly}, Lemma~\ref{lem-roots}, Proposition~\ref{prop-rationality} and Lemma~\ref{lem-0},
we can prove Theorem~\ref{thm-1}.
\begin{proof}[Proof of Theorem~\ref{thm-1}]
Since $D_n(a)$ is a polynomial in $q^a$ with degree at most $nb$ by Corollary~\ref{cor-poly},
we can determine it by its $nb+1$ distinct values.
Recall that $R$ (defined in \eqref{def-roots}) is the collection of all the roots of $D_n(a)$.
Assume $c\geq b$ in the proof below, then the number of roots $|R|=nb$.
We use the values $-a\in R\cup\{0\}$ to determine $D_n(a)$. By Lemma~\ref{lem-roots},
$D_n(a)=0$ for $-a\in R$. Then we can write
\[
D_n(a)
=\prod_{i\in R}\frac{1-q^{a+i}}{1-q^i}D_n(0).
\]
Using \eqref{e-0} and Lemma~\ref{lem-0}, we have
\begin{equation}\label{e-Dn}
D_n(a)
=\prod_{i\in R}\frac{1-q^{a+i}}{1-q^i}\cdot
\frac{1-q^{n_k(c+1)}}{1-q^{c+1}}\qbinom{nc+n_k-1}{c}
D_{n-1}\big((n_0,\dots,n_{k-1},n_{k}-1,n_{k+1},\dots,n_p);0,0,c\big).
\end{equation}
We can derive a closed-form expression for $D_n(a)$ using \eqref{e-Dn} and the recursion
for $D_n(0)$ outlined in \eqref{e-rec0}. However, this expression is not particularly elegant and requires additional notation.
Therefore, we aim to establish a recursion for $D_n(a)$ in terms of
$D_{n-1}(a)$. We present the transformation below.

Notice that $D_{n-1}\big((n_0,\dots,n_{k-1},n_{k}-1,n_{k+1},\dots,n_p);a,b,c\big)$
is a polynomial in $q^a$ with degree at most $(n-1)b$ by Corollary~\ref{cor-poly},
and it vanishes when $-a\in R\setminus R^p_{n-1}$
by Lemma~\ref{lem-roots}. Then
\begin{align}
&D_{n-1}\big((n_0,\dots,n_{k-1},n_{k}-1,n_{k+1},\dots,n_p);a,b,c\big)\nonumber \\
&=\prod_{i\in R\setminus R^p_{n-1}}\frac{1-q^{a+i}}{1-q^i}
D_{n-1}\big((n_0,\dots,n_{k-1},n_{k}-1,n_{k+1},\dots,n_p);0,b,c\big)\nonumber \\
&=\prod_{i\in R\setminus R^p_{n-1}}\frac{1-q^{a+i}}{1-q^i}
D_{n-1}\big((n_0,\dots,n_{k-1},n_{k}-1,n_{k+1},\dots,n_p);0,0,c\big)\label{e-Dn-1}.
\end{align}
Here the last equality holds by homogeneity.
Substituting \eqref{e-Dn-1} into \eqref{e-Dn} gives
\begin{equation}\label{e-Dn2}
D_n(a)=\prod_{i\in R^p_{n-1}}\frac{1-q^{a+i}}{1-q^i}
\frac{1-q^{n_k(c+1)}}{1-q^{c+1}}\qbinom{nc+n_k-1}{c}
D_{n-1}\big((n_0,\dots,n_{k-1},n_{k}-1,n_{k+1},\dots,n_p);a,b,c\big).
\end{equation}
By the definition of the $R^j_i$ in \eqref{R_p} for $i=n-1$ and $n_k=\max\{n_1,\dots,n_p\}=n_{s_p}$,
we can write
\[
R^p_{n-1}=\{(n-1)c+n_k,(n-1)c+n_k+1,\dots,(n-1)c+n_k+b-1\}.
\]
Then we can rewrite \eqref{e-Dn2} as
\begin{equation}\label{e-Dn3}
D_n(a)=\frac{(1-q^{n_k(c+1)})(q^{a+(n-1)c+n_k})_b}{(1-q^{c+1})(q^{(n-1)c+n_k})_b}
\qbinom{nc+n_k-1}{c}D_{n-1}\big((n_0,\dots,n_{k-1},n_{k}-1,n_{k+1},\dots,n_p);a,b,c\big).
\end{equation}
We complete the proof by extending \eqref{e-Dn3} to all nonnegative integers $c$
using Proposition~\ref{prop-rationality}.
\end{proof}

\section{The splitting formula and consequences}\label{sec-split}

In this section, we obtain a splitting formula \eqref{e-sp-main} for the rational function $S\big((n_0,\dots,n_p),c;x,y\big)$ defined in \eqref{def-L} below and provide its several consequences.
In Subsection~\ref{sec-SP0}, we present the splitting formula.
In Subsection~\ref{sec-SP1}, we obtain a recursion for $D_n(0)$ by the splitting formula.
Hence, we complete the proof of Lemma~\ref{lem-0}.
In Subsection~\ref{sec-SP2}, we characterize a family of vanishing coefficients of
$L_{n_0,n_1,\dots,n_p}(0,0,c;x)$ (defined in \eqref{L}).
These vanishing coefficients are crucial in determining many roots of $D_n(a)$ that the classic Gessel--Xin method is invalid.

\subsection{The splitting formula}\label{sec-SP0}
Let
\begin{equation}\label{def-L}
S=S\big((n_0,\dots,n_p),c;\x,y\big)
:=\frac{\prod_{1\leq i<j\leq n}(x_i/x_j)_{c+\varepsilon_{ij}}(qx_j/x_i)_{c+\varepsilon_{ij}}}
{\prod_{i=0}^{k-1}\prod_{l\in N_i}(y/x_l)_{c}
\prod_{l\in N_k}(q^{-1}y/x_l)_{c+1}
\prod_{i=k+1}^{p}\prod_{l\in N_i}(q^{-1}y/x_l)_{c}}.
\end{equation}
Here $k$ is also an integer such that $n_k=\max\{n_1,\dots,n_p\}$, the $N_i$ and the $\varepsilon_{ij}$
are defined in \eqref{def-Tij}, and for $z\in \mathbb{Z}$ all forms of $(1-q^zy/x_l)^{-1}$ in
\eqref{def-L} are explained as
\begin{equation}\label{e-explain}
\frac{1}{1-q^zy/x_l}=\sum_{j\geq 0}(q^zy/x_l)^j.
\end{equation}
We can inteprete \eqref{e-explain} as:
since the $z$ only ranges over a finite set, we can
find a parameter $y$ such that all forms of $|q^zy/x_l|<1$.
Then \eqref{e-explain} holds.
Hence, the function $S$ in \eqref{def-L} is in fact regarded as a Laurent series in each $x_l^{-1}$ and a power series in $y$,
rather than a rational function.
Note that by the above explanation it is not hard to see that
\begin{equation}\label{e-sp-1}
\CT_{y,\x}S\big((n_0,\dots,n_p),c;x,y\big)
=\CT_{\x} \prod_{1\leq i<j\leq n}(x_i/x_j)_{c+\varepsilon_{ij}}(qx_j/x_i)_{c+\varepsilon_{ij}}
=D_n(0).
\end{equation}
We find that $S$ can be written as
\begin{equation}\label{e-sp-0}
S= \sum_{z,l}\frac{A_{z,l}(x_l^{-1})}{1-q^zy/x_l}
\end{equation}
using the partial fraction decomposition,
where the sum is a finite summation and
each $A_{z,l}(x_l^{-1})$ is a polynomial in $x_{l}^{-1}$.
Then the constant term of each summand in \eqref{e-sp-0}
\[
\CT_{y,x_l}\frac{A_{z,l}(x_l^{-1})}{1-q^zy/x_l}=\CT_{x_l}A_{z,l}(x_l^{-1})
=A_{z,l}(x_l^{-1})\big|_{x_l^{-1}=0}=A_{z,l}(0).
\]
Together with \eqref{e-sp-1}, we have
\[
D_n(0)=\sum_{z,l}\CT_{x^{(l)}}A_{z,l}(0),
\]
where $x^{(l)}=(x_1,\dots,x_{l-1},x_{l+1},\dots,x_n)$.
We will show that each $A_{z,l}(0)$ is a Laurent polynomial in $x^{(l)}$.
The above operation originates from the splitting idea of
Cai's proof \cite{cai} of the $q$-Dyson constant term identity.
The main difficulty to utilize this idea is to find a suitable denominator such as that in \eqref{def-L}.

The main result of this section is the next lemma.
\begin{lem}\label{lem-split}
Let $S$ and $k$ be defined as in \eqref{def-L}.
Denote by $\sigma_l=n_0+n_1+\cdots+n_l$ for $l=0,\dots,p$.
We find that $S$ admits the next splitting formula:
\begin{equation}\label{e-sp-main}
S\big((n_0,\dots,n_p),c;\x,y\big)=\sum_{t=0}^{k-1}\sum_{i\in N_t}\sum_{j=0}^{c-1}\frac{A_{ij}^{(t)}}{1-q^jy/x_i}
+\sum_{i\in N_k}\sum_{j=-1}^{c-1}\frac{A^{(k)}_{ij}}{1-q^jy/x_i}
+\sum_{t=k+1}^{p}\sum_{i\in N_t}\sum_{j=-1}^{c-2}\frac{A_{ij}^{(t)}}{1-q^jy/x_i},
\end{equation}
where
\begin{subequations}
\begin{align}\label{A-0}
A^{(0)}_{ij}&=\frac{\prod_{\substack{1\leq u<v\leq n\\u,v\neq i}}(x_u/x_v)_{c+\varepsilon_{uv}}(qx_v/x_u)_{c+\varepsilon_{uv}}}
{(q^{-j})_j(q)_{c-j-1}}
\prod_{l=1}^{i-1}q^{c(j+1)}\big(q^{-c}x_l/x_i\big)_{j+1}\big(q^{j+1}x_l/x_i\big)_{c-j-1}
 \\
&\times\prod_{l=i+1}^{\sigma_{k-1}}q^{cj}\big(q^{1-c}x_l/x_i\big)_j\big(q^{j+1}x_l/x_i\big)_{c-j}
\prod_{l=\sigma_{k-1}+1}^{\sigma_k}q^{(c+1)j+1}\boldsymbol{(-x_l/x_i)}(q^{1-c}x_l/x_i)_j(q^{j+2}x_l/x_i)_{c-j-1}\nonumber \\
&\times \prod_{l=\sigma_k+1}^nq^{c(j+1)}(q^{1-c}x_l/x_i)_{j+1}(q^{j+2}x_l/x_i)_{c-j-1},\nonumber
\end{align}
\begin{align}\label{A-1}
A^{(t)}_{ij}&=\frac{\prod_{\substack{1\leq u<v\leq n\\u,v\neq i}}(x_u/x_v)_{c+\varepsilon_{uv}}(qx_v/x_u)_{c+\varepsilon_{uv}}}
{(q^{-j})_j(q)_{c-j-1}}
\prod_{l=1}^{\sigma_{t-1}}q^{c(j+1)}\big(q^{-c}x_l/x_i\big)_{j+1}\big(q^{j+1}x_l/x_i\big)_{c-j-1}
 \\
&\times\prod_{l=\sigma_{t-1}+1}^{i-1}q^{c(j+2)+1}\boldsymbol{(-x_i/x_l)}\big(q^{-c-1}x_l/x_i\big)_{j+2}
\big(q^{j+1}x_l/x_i\big)_{c-j}\nonumber \\
&\times\prod_{l=i+1}^{\sigma_t}q^{c(j+1)}\boldsymbol{(-x_i/x_l)}(q^{-c}x_l/x_i)_{j+1}(q^{j+1}x_l/x_i)_{c+1-j}\nonumber \\
&\times \prod_{l=\sigma_t+1}^{\sigma_{k-1}}q^{cj}(q^{1-c}x_l/x_i)_{j}(q^{j+1}x_l/x_i)_{c-j}
\prod_{l=\sigma_{k-1}+1}^{\sigma_k}q^{(c+1)j+1}\boldsymbol{(-x_l/x_i)}(q^{1-c}x_l/x_i)_{j}(q^{j+2}x_l/x_i)_{c-j-1}\nonumber \\
&\times \prod_{l=\sigma_{k}+1}^{n}q^{c(j+1)}(q^{1-c}x_l/x_i)_{j+1}(q^{j+2}x_l/x_i)_{c-j-1}\nonumber
\end{align}
for $t=1,\dots,k-1$,
\begin{align}\label{A-2}
A^{(k)}_{ij}&=\frac{\prod_{\substack{1\leq u<v\leq n\\u,v\neq i}}(x_u/x_v)_{c+\varepsilon_{uv}}(qx_v/x_u)_{c+\varepsilon_{uv}}}
{(q^{-j-1})_{j+1}(q)_{c-j-1}}
\prod_{l=1}^{\sigma_{k-1}}q^{c(j+1)}\big(q^{-c}x_l/x_i\big)_{j+1}\big(q^{j+1}x_l/x_i\big)_{c-j-1}
 \\
&\times\prod_{l=\sigma_{k-1}+1}^{i-1}q^{(c+1)(j+2)}\big(q^{-c-1}x_l/x_i\big)_{j+2}
\big(q^{j+2}x_l/x_i\big)_{c-j-1}\nonumber \\
&\times\prod_{l=i+1}^{\sigma_k}q^{(c+1)(j+1)}(q^{-c}x_l/x_i)_{j+1}(q^{j+2}x_l/x_i)_{c-j}
\prod_{l=\sigma_k+1}^{n}q^{c(j+1)}(q^{1-c}x_l/x_i)_{j+1}(q^{j+2}x_l/x_i)_{c-j-1},\nonumber
\end{align}
and
\begin{align}\label{A-3}
A^{(t)}_{ij}&=\frac{\prod_{\substack{1\leq u<v\leq n\\u,v\neq i}}(x_u/x_v)_{c+\varepsilon_{uv}}(qx_v/x_u)_{c+\varepsilon_{uv}}}
{(q^{-j-1})_{j+1}(q)_{c-j-2}}
\prod_{l=1}^{\sigma_{k-1}}q^{c(j+1)}\big(q^{-c}x_l/x_i\big)_{j+1}\big(q^{j+1}x_l/x_i\big)_{c-j-1}
 \\
&\times\prod_{l=\sigma_{k-1}+1}^{\sigma_k}q^{(c+1)(j+1)}\boldsymbol{(-x_l/x_i)}\big(q^{-c}x_l/x_i\big)_{j+1}
\big(q^{j+2}x_l/x_i\big)_{c-j-2}\nonumber \\
&\times\prod_{l=\sigma_k+1}^{\sigma_{t-1}}q^{c(j+2)}(q^{-c}x_l/x_i)_{j+2}(q^{j+2}x_l/x_i)_{c-j-2}\nonumber \\
&\times \prod_{l=\sigma_{t-1}+1}^{i-1}q^{c(j+3)+1}\boldsymbol{(-x_i/x_l)}(q^{-c-1}x_l/x_i)_{j+3}(q^{j+2}x_l/x_i)_{c-j-1}\nonumber \\
&\times\prod_{l=i+1}^{\sigma_t}q^{c(j+2)}\boldsymbol{(-x_i/x_l)}(q^{-c}x_l/x_i)_{j+2}(q^{j+2}x_l/x_i)_{c-j}\nonumber \\
&\times \prod_{l=\sigma_{t}+1}^{n}q^{c(j+1)}\big(q^{1-c}x_l/x_i\big)_{j+1}\big(q^{j+2}x_l/x_i\big)_{c-j-1}\nonumber
\end{align}
for $t=k+1,\dots,p$.
\end{subequations}
\end{lem}
Note that $A_{ij}^{(0)}$, $A_{ij}^{(t)}(t=1,\dots,k-1)$, $A_{ij}^{(k)}$
and $A_{ij}^{(t)}(t=k+1,\dots,p)$ are polynomials in $x_i^{-1}$
with degrees at least $n_k$, $n_k-n_t+1$, $0$ and $n_k-n_t+1$ respectively.
These can be seen by the bold symbols in the above.
The superscripts $(t)$ of the $A_{ij}^{(t)}$ are in fact determined by their corresponding subscripts $i$. We use the superscripts $(t)$ to identify the relation with
the subsets $N_t$.

Before giving the proof of Lemma~\ref{lem-split}, we need the following simple results.
\begin{lem}
Let $i,j$ be nonnegative integers. Then,
\begin{subequations}\label{e-ab}
\begin{equation}\label{prop-b1}
\frac{(1/y)_{i}(qy)_j}{(q^{-t}/y)_i}=q^{it}(q^{1-i}y)_t(q^{t+1}y)_{j-t}
\quad \text{for $t=0,\dots,j$};
\end{equation}
\begin{equation}\label{prop-b2}
\frac{(y)_j(q/y)_i}{(q^{-t}/y)_i}=q^{i(t+1)}(q^{-i}y)_{t+1}(q^{t+1}y)_{j-t-1}
\quad \text{for $t=-1,\dots,j-1$};
\end{equation}
\begin{equation}\label{prop-c}
\frac{(y)_j(q/y)_i}{(q^{-t}/y)_{i+1}}=-q^{(i+1)t}y(q^{-i}y)_t(q^{t+1}y)_{j-t-1}
\quad \text{for $t=0,\dots,j-1$};
\end{equation}
\begin{equation}\label{prop-d}
\frac{(1/y)_{i+1}(qy)_{j+1}}{(q^{-t}/y)_i}=-q^{i(t+1)}y^{-1}(q^{-i}y)_{t+1}(q^{t+1}y)_{j+1-t}
\quad \text{for $t=-1,\dots,j+1$};
\end{equation}
\begin{equation}\label{prop-e}
\frac{(y)_{j+1}(q/y)_{i+1}}{(q^{-t}/y)_i}=-q^{i(t+2)+1}y^{-1}(q^{-i-1}y)_{t+2}(q^{t+1}y)_{j-t}
\quad \text{for $t=-2,\dots,j$}.
\end{equation}
\end{subequations}
\end{lem}
Note that \eqref{prop-c} only holds for a positive integer $j$.
\begin{proof}
For $0\leq t\leq j$,
\[
\frac{(1/y)_{i}(qy)_j}{(q^{-t}/y)_i}=\frac{(q^{i-t}/y)_t(qy)_j}{(q^{-t}/y)_t}
=\frac{(-1/y)^tq^{it-\binom{t+1}{2}}(q^{1-i}y)_t(qy)_j}{(-1/y)^{t}q^{-\binom{t+1}{2}}(qy)_t}
=q^{it}(q^{1-i}y)_t(q^{t+1}y)_{j-t}.
\]

Taking $y\mapsto y/q$ and $t\mapsto t+1$ in \eqref{prop-b1} yields
\eqref{prop-b2} for $-1\leq t\leq j-1$.

For $0\leq t\leq j-1$ and $j$ a positive integer,
\[
\frac{(y)_j(q/y)_i}{(q^{-t}/y)_{i+1}}=
\frac{(y)_j(q^{i-t+1}/y)_t}{(q^{-t}/y)_{t+1}}=
\frac{(y)_j(-1/y)^tq^{it-\binom{t}{2}}(q^{-i}y)_t}{(-1/y)^{t+1}q^{-\binom{t+1}{2}}(y)_{t+1}}=
-yq^{(i+1)t}(q^{-i}y)_t(q^{t+1}y)_{j-t-1}.
\]

For $-1\leq t\leq j+1$,
\begin{multline*}
\frac{(1/y)_{i+1}(qy)_{j+1}}{(q^{-t}/y)_i}
=\frac{(q^{i-t}/y)_{t+1}(qy)_{j+1}}{(q^{-t}/y)_t}
=\frac{(-1/y)^{t+1}q^{i(t+1)-\binom{t+1}{2}}(q^{-i}y)_{t+1}(qy)_{j+1}}
{(-1/y)^tq^{-\binom{t+1}{2}}(qy)_t}\\
=-q^{i(t+1)}y^{-1}(q^{-i}y)_{t+1}(q^{t+1}y)_{j+1-t}.
\end{multline*}

Taking $y\mapsto y/q$ and $t\mapsto t+1$ in \eqref{prop-d} yields
\eqref{prop-e} for $-2\leq t\leq j$.
\end{proof}

\begin{proof}[Proof of Lemma~\ref{lem-split}]
By partial fraction decomposition of $S$ with respect to $y$,
we can rewrite $S$ as \eqref{e-sp-main} and
\begin{subequations}
\begin{equation}\label{Aij0}
A^{(0)}_{ij}=S\cdot (1-q^jy/x_i)|_{y=q^{-j}x_i} \quad \text{for $i=1,\dots,n_0$ and $j=0,\dots,c-1$,}
\end{equation}
\begin{equation}\label{Aij1}
A^{(t)}_{ij}=S\cdot (1-q^jy/x_i)|_{y=q^{-j}x_i} \quad \text{for $t=1,\dots,k-1$, $i=n_0+1,\dots,\sigma_{k-1}$ and $j=0,\dots,c-1$,}
\end{equation}
\begin{equation}\label{Aij2}
A^{(k)}_{ij}=S\cdot (1-q^jy/x_i)|_{y=q^{-j}x_i} \quad \text{for $i=\sigma_{k-1}+1,\dots,\sigma_{k}$ and $j=-1,\dots,c-1$,}
\end{equation}
and
\begin{equation}\label{Aij3}
A^{(t)}_{ij}=S\cdot (1-q^jy/x_i)|_{y=q^{-j}x_i} \quad \text{for $t=k+1,\dots,p$, $i=\sigma_{k}+1,\dots,n$ and $j=-1,\dots,c-2$.}
\end{equation}
\end{subequations}
Note that we can write $N_i=\{\sigma_{i-1}+1,\sigma_{i-1}+2,\dots,\sigma_i\}$ for $i=0,\dots,p$. Here we take $\sigma_{-1}:=0$.
Hence, for example, the condition $i=n_0+1,\dots,\sigma_{k-1}$ in the above is the same as $i\in N_1\cup \dots \cup N_{k-1}$.

Carrying out the substitution $y=q^{-j}x_i$ in $S\cdot(1-q^jy/x_i)$ for $i=1,\dots,n_0$ yields
\begin{multline}\label{A0}
A^{(0)}_{ij}=
\frac{1}{(q^{-j})_j(q)_{c-j-1}}
\prod_{l=1}^{i-1}\frac{(x_l/x_i)_{c}(qx_i/x_l)_{c}}{(q^{-j}x_i/x_l)_{c}}
\prod_{l=i+1}^{\sigma_{k-1}}\frac{(x_i/x_l)_{c}(qx_l/x_i)_{c}}{(q^{-j}x_i/x_l)_{c}}
\\
\times
\prod_{l=\sigma_{k-1}+1}^{\sigma_k}\frac{(x_i/x_l)_{c}(qx_l/x_i)_{c}}{(q^{-j-1}x_i/x_l)_{c+1}}
\prod_{l=\sigma_k+1}^{n}\frac{(x_i/x_l)_{c}(qx_l/x_i)_{c}}{(q^{-j-1}x_i/x_l)_{c}}
\prod_{\substack{1\leq v<u\leq n\\v,u\neq i}}
(x_v/x_u)_{c+\varepsilon_{uv}}(qx_u/x_v)_{c+\varepsilon_{uv}}.
\end{multline}
Using \eqref{prop-b2} with $(i,j,t,y)\mapsto (c,c,j,x_l/x_i)$,
we have
\begin{subequations}\label{e-B-2}
\begin{equation}\label{B1-2}
\frac{(x_l/x_i)_{c}(qx_i/x_l)_{c}}{(q^{-j}x_i/x_l)_{c}}
=q^{c(j+1)}\big(q^{-c}x_l/x_i\big)_{j+1}\big(q^{j+1}x_l/x_i\big)_{c-j-1} \quad \text{for $j=-1,\dots,c-1$}.
\end{equation}
Using \eqref{prop-b1} with $(i,j,t,y)\mapsto (c,c,j,x_l/x_i)$ gives
\begin{equation}\label{B2-2}
\frac{(x_i/x_l)_{c}(qx_l/x_i)_{c}}{(q^{-j}x_i/x_l)_{c}}
=q^{cj}\big(q^{1-c}x_l/x_i\big)_j\big(q^{j+1}x_l/x_i\big)_{c-j}
\quad \text{for $j=0,\dots,c$}.
\end{equation}
Using \eqref{prop-c} with $(i,j,t,y)\mapsto (c,c,j,qx_l/x_i)$ yields
\begin{equation}\label{B2-2-2}
\frac{(x_i/x_l)_{c}(qx_l/x_i)_{c}}{(q^{-j-1}x_i/x_l)_{c+1}}=
-q^{(c+1)j+1}x_l/x_i(q^{1-c}x_l/x_i)_j(q^{j+2}x_l/x_i)_{c-j-1} \quad \text{for $j=0,\dots,c-1$}.
\end{equation}
Using \eqref{prop-b1} with $(i,j,t,y)\mapsto (c,c,j+1,x_l/x_i)$ gives
\begin{equation}\label{B2-2-3}
\frac{(x_i/x_l)_{c}(qx_l/x_i)_{c}}{(q^{-j-1}x_i/x_l)_{c}}
=q^{c(j+1)}\big(q^{1-c}x_l/x_i\big)_{j+1}\big(q^{j+2}x_l/x_i\big)_{c-j-1}
\quad \text{for $j=-1,\dots,c-1$}.
\end{equation}
\end{subequations}
Substituting \eqref{e-B-2} into \eqref{A0},
we obtain \eqref{A-0}.

Carrying out the substitution $y=q^{-j}x_i$ in $S\cdot (1-q^{-j}y/x_i)$ for $i=n_0+1,\dots,\sigma_{k-1}$ and $t=1,\dots,k-1$ yields
\begin{align}\label{BB2}
A^{(t)}_{ij}&=\frac{1}{(q^{-j})_j(q)_{c-j-1}}
\prod_{l=1}^{\sigma_{t-1}}\frac{(x_l/x_i)_{c}(qx_i/x_l)_{c}}{(q^{-j}x_i/x_l)_{c}}
\prod_{l=\sigma_{t-1}+1}^{i-1}\frac{(x_l/x_i)_{c+1}(qx_i/x_l)_{c+1}}{(q^{-j}x_i/x_l)_{c}}
\\
&\quad \times
\prod_{l=i+1}^{\sigma_t}\frac{(x_i/x_l)_{c+1}(qx_l/x_i)_{c+1}}{(q^{-j}x_i/x_l)_c}
\prod_{l=\sigma_t+1}^{\sigma_{k-1}}\frac{(x_i/x_l)_{c}(qx_l/x_i)_{c}}{(q^{-j}x_i/x_l)_c}
\prod_{l=\sigma_{k-1}+1}^{\sigma_{k}}\frac{(x_i/x_l)_{c}(qx_l/x_i)_{c}}{(q^{-j-1}x_i/x_l)_{c+1}}\nonumber \\
&\quad \times \prod_{l=\sigma_{k}+1}^{n}\frac{(x_i/x_l)_{c}(qx_l/x_i)_{c}}{(q^{-j-1}x_i/x_l)_{c}}
\prod_{\substack{1\leq v<u\leq n\\v,u\neq i}}
(x_v/x_u)_{c+\varepsilon_{uv}}(qx_u/x_v)_{c+\varepsilon_{uv}}.\nonumber
\end{align}
Using \eqref{prop-e} with $(i,j,t,y)\mapsto (c,c,j,x_l/x_i)$ gives
\begin{subequations}\label{e-At}
\begin{equation}\label{e-At-1}
\frac{(x_l/x_i)_{c+1}(qx_i/x_l)_{c+1}}{(q^{-j}x_i/x_l)_c}
=-q^{c(j+2)+1}x_i/x_l(q^{-c-1}x_l/x_i)_{j+2}(q^{j+1}x_l/x_i)_{c-j}
\end{equation}
for $-2\leq j\leq c$.
Using \eqref{prop-d} with $(i,j,t,y)\mapsto (c,c,j,x_l/x_i)$ gives
\begin{equation}\label{e-At-2}
\frac{(x_i/x_l)_{c+1}(qx_l/x_i)_{c+1}}{(q^{-j}x_i/x_l)_c}
=-q^{c(j+1)}x_i/x_l(q^{-c}x_l/x_i)_{j+1}(q^{j+1}x_l/x_i)_{c+1-j}
\end{equation}
\end{subequations}
for $-1\leq j\leq c+1$.
Substituting \eqref{B1-2}, \eqref{e-At-1}, \eqref{e-At-2}, \eqref{B2-2},
\eqref{B2-2-2} and \eqref{B2-2-3} into \eqref{BB2} respectively, we obtain \eqref{A-1}.

Carrying out the substitution $y=q^{-j}x_i$ in $S\cdot (1-q^{-j}y/x_i)$ for $i=\sigma_{k-1}+1,\dots,\sigma_k$ yields
\begin{align}\label{Ak}
A^{(k)}_{ij}&=\frac{1}{(q^{-j-1})_{j+1}(q)_{c-j-1}}
\prod_{l=1}^{\sigma_{k-1}}\frac{(x_l/x_i)_{c}(qx_i/x_l)_{c}}{(q^{-j}x_i/x_l)_{c}}
\prod_{l=\sigma_{k-1}+1}^{i-1}\frac{(x_l/x_i)_{c+1}(qx_i/x_l)_{c+1}}{(q^{-1-j}x_i/x_l)_{c+1}}
\\
&\quad \times
\prod_{l=i+1}^{\sigma_k}\frac{(x_i/x_l)_{c+1}(qx_l/x_i)_{c+1}}{(q^{-1-j}x_i/x_l)_{c+1}}
\prod_{l=\sigma_k+1}^{n}\frac{(x_i/x_l)_{c}(qx_l/x_i)_{c}}{(q^{-1-j}x_i/x_l)_c}
\prod_{\substack{1\leq v<u\leq n\\v,u\neq i}}
(x_v/x_u)_{c+\varepsilon_{uv}}(qx_u/x_v)_{c+\varepsilon_{uv}}.\nonumber
\end{align}
Using \eqref{B1-2} with $(c,j)\mapsto (c+1,j+1)$ yields
\begin{equation}\label{B1-2+1}
\frac{(x_l/x_i)_{c+1}(qx_i/x_l)_{c+1}}{(q^{-j-1}x_i/x_l)_{c+1}}
=q^{(c+1)(j+2)}\big(q^{-c-1}x_l/x_i\big)_{j+2}\big(q^{j+2}x_l/x_i\big)_{c-j-1} \quad \text{for $j=-2,\dots,c-1$}.
\end{equation}
Using \eqref{B2-2} with $(c,j)\mapsto (c+1,j+1)$ gives
\begin{equation}\label{B2-2+1}
\frac{(x_i/x_l)_{c+1}(qx_l/x_i)_{c+1}}{(q^{-j-1}x_i/x_l)_{c+1}}
=q^{(c+1)(j+1)}\big(q^{-c}x_l/x_i\big)_{j+1}\big(q^{j+2}x_l/x_i\big)_{c-j}
\quad \text{for $j=-1,\dots,c$}.
\end{equation}
Substituting \eqref{B1-2}, \eqref{B1-2+1}, \eqref{B2-2+1} and \eqref{B2-2-3}
into \eqref{Ak} respectively yields \eqref{A-2}.

Carrying out the substitution $y=q^{-j}x_i$ in $S\cdot (1-q^{-j}y/x_i)$ for $i=\sigma_k+1,\dots,n$ and $t=k+1,\dots,p$ yields
\begin{align}\label{e-At4}
A^{(t)}_{ij}&=\frac{1}{(q^{-1-j})_{j+1}(q)_{c-j-2}}
\prod_{l=1}^{\sigma_{k-1}}\frac{(x_l/x_i)_{c}(qx_i/x_l)_{c}}{(q^{-j}x_i/x_l)_{c}}
\prod_{l=\sigma_{k-1}+1}^{\sigma_k}\frac{(x_l/x_i)_{c}(qx_i/x_l)_{c}}{(q^{-1-j}x_i/x_l)_{c+1}}
\\
&\quad \times
\prod_{l=\sigma_k+1}^{\sigma_{t-1}}\frac{(x_l/x_i)_{c}(qx_i/x_l)_{c}}{(q^{-1-j}x_i/x_l)_{c}}
\prod_{l=\sigma_{t-1}+1}^{i-1}\frac{(x_l/x_i)_{c+1}(qx_i/x_l)_{c+1}}{(q^{-1-j}x_i/x_l)_{c}}
\prod_{l=i+1}^{\sigma_{t}}\frac{(x_i/x_l)_{c+1}(qx_l/x_i)_{c+1}}{(q^{-j-1}x_i/x_l)_{c}}\nonumber \\
&\quad \times \prod_{l=\sigma_{t}+1}^{n}\frac{(x_i/x_l)_{c}(qx_l/x_i)_{c}}{(q^{-j-1}x_i/x_l)_{c}}
\prod_{\substack{1\leq v<u\leq n\\v,u\neq i}}
(x_v/x_u)_{c+\varepsilon_{uv}}(qx_u/x_v)_{c+\varepsilon_{uv}}.\nonumber
\end{align}
Using \eqref{prop-c} with $(i,j,t,y)\mapsto (c,c,j+1,x_l/x_i)$ yields
\begin{equation}\label{A-t-3}
\frac{(x_l/x_i)_c(qx_i/x_l)_c}{(q^{-1-j}x_i/x_l)_{c+1}}
=-q^{(c+1)(j+1)}x_l/x_i(q^{-c}x_l/x_i)_{j+1}(q^{j+2}x_l/x_i)_{c-j-2}
\quad \text{for $j=-1,\dots,c-2$}.
\end{equation}
Using \eqref{B1-2} with $j\mapsto j+1$ gives
\begin{equation}\label{B1-2+2}
\frac{(x_l/x_i)_{c}(qx_i/x_l)_{c}}{(q^{-1-j}x_i/x_l)_{c}}
=q^{c(j+2)}\big(q^{-c}x_l/x_i\big)_{j+2}\big(q^{j+2}x_l/x_i\big)_{c-j-2} \quad \text{for $j=-2,\dots,c-2$}.
\end{equation}
Using \eqref{e-At-1} with $j\mapsto j+1$ gives
\begin{equation}\label{e-At-4}
\frac{(x_l/x_i)_{c+1}(qx_i/x_l)_{c+1}}{(q^{-j-1}x_i/x_l)_c}
=-q^{c(j+3)+1}x_i/x_l(q^{-c-1}x_l/x_i)_{j+3}(q^{j+2}x_l/x_i)_{c-j-1}
\quad \text{for $j=-3,\dots,c-1$}.
\end{equation}
Using \eqref{e-At-2} with $j\mapsto j+1$ gives
\begin{equation}\label{e-At-5}
\frac{(x_i/x_l)_{c+1}(qx_l/x_i)_{c+1}}{(q^{-j-1}x_i/x_l)_c}
=-q^{c(j+2)}x_i/x_l(q^{-c}x_l/x_i)_{j+2}(q^{j+2}x_l/x_i)_{c-j}
\quad \text{for $j=-2,\dots,c$}.
\end{equation}
Substituting \eqref{B1-2}, \eqref{A-t-3}, \eqref{B1-2+2},\eqref{e-At-4}, \eqref{e-At-5}
and \eqref{B2-2-3} into \eqref{e-At4} respectively yields \eqref{A-3}.
\end{proof}

\subsection{The recursion for $D_n(0)$}\label{sec-SP1}

Using the splitting formula \eqref{e-sp-main} for $S$, we can obtain a recursion for $D_n(0)$.
That is Lemma~\ref{lem-0}. In this subsection, we give a proof of this lemma.

\begin{proof}[Proof of Lemma~\ref{lem-0}]
Using \eqref{e-sp-1} and \eqref{e-sp-main}, we obtain
\begin{align}
D_n(0)&=\sum_{t=0}^{k-1}\sum_{i\in N_t}\sum_{j=0}^{c-1}\CT_{y,x}\frac{A_{ij}^{(t)}}{1-q^jy/x_i}
+\sum_{i\in N_k}\sum_{j=-1}^{c-1}\CT_{y,x}\frac{A^{(k)}_{ij}}{1-q^jy/x_i}
+\sum_{t=k+1}^{p}\sum_{i\in N_t}\sum_{j=-1}^{c-2}\CT_{y,x}\frac{A_{ij}^{(t)}}{1-q^jy/x_i}\nonumber \\
&=\sum_{t=0}^{k-1}\sum_{i\in N_t}\sum_{j=0}^{c-1}\CT_{x}A_{ij}^{(t)}
+\sum_{i\in N_k}\sum_{j=-1}^{c-1}\CT_{x}A^{(k)}_{ij}
+\sum_{t=k+1}^{p}\sum_{i\in N_t}\sum_{j=-1}^{c-2}\CT_{x}A_{ij}^{(t)}.\label{rec-Dn-0}
\end{align}
Note that $n_k=\max\{n_1,\dots,n_p\}$.
Since $A_{ij}^{(0)}$, $A_{ij}^{(t)}(t\in \{1,\dots,p\}\setminus \{k\})$ and $A_{ij}^{(k)}$
are polynomials in $x_i^{-1}$
with degrees at least $n_k$, $n_k-n_t+1$ and $0$ respectively,
the above equation \eqref{rec-Dn-0} reduces to
\begin{equation*}\label{rec-Dn-1}
D_n(0)=\sum_{i\in N_k}\sum_{j=-1}^{c-1}\CT_{x}A^{(k)}_{ij}.
\end{equation*}
By the expression for $A^{(k)}_{ij}$ in \eqref{A-2}, we have
\begin{multline}\label{rec-Dn-2}
D_n(0)=\sum_{i\in N_k}\sum_{j=-1}^{c-1}\CT_{x}
\frac{\prod_{\substack{1\leq u<v\leq n\\u,v\neq i}}(x_u/x_v)_{c+\varepsilon_{uv}}(qx_v/x_u)_{c+\varepsilon_{uv}}}
{(q^{-j-1})_{j+1}(q)_{c-j-1}}\\
\times\prod_{l=1}^{\sigma_{k-1}}q^{c(j+1)}
\prod_{l=\sigma_{k-1}+1}^{i-1}q^{(c+1)(j+2)}
\prod_{l=i+1}^{\sigma_k}q^{(c+1)(j+1)}
\prod_{l=\sigma_k+1}^{n}q^{c(j+1)}.
\end{multline}
It is known that $N_k=\{\sigma_{k-1}+1,\dots,\sigma_k\}$ and
\[
\CT_{\x}\prod_{\substack{1\leq u<v\leq n\\u,v\neq i}}(x_u/x_v)_{c+\varepsilon_{uv}}(qx_v/x_u)_{c+\varepsilon_{uv}}
=D_{n-1}\big((n_0,\dots,n_{k-1},n_{k}-1,n_{k+1},\dots,n_p);0,0,c\big)
\]
for $i\in N_k$.
Then, compare \eqref{rec-Dn-2} with \eqref{e-rec0} in Lemma~\ref{lem-0},
we prove the lemma by showing that
\begin{multline}\label{rec-Dn-3}
\sum_{i=\sigma_{k-1}+1}^{\sigma_k}\sum_{j=-1}^{c-1}
\frac{1}{(q^{-j-1})_{j+1}(q)_{c-j-1}}\prod_{l=1}^{\sigma_{k-1}}q^{c(j+1)}
\prod_{l=\sigma_{k-1}+1}^{i-1}q^{(c+1)(j+2)}
\prod_{l=i+1}^{\sigma_k}q^{(c+1)(j+1)}
\prod_{l=\sigma_k+1}^{n}q^{c(j+1)}\\
=\frac{1-q^{n_k(c+1)}}{1-q^{c+1}}\qbinom{nc+n_k-1}{c}.
\end{multline}
Denote the left-hand side of \eqref{rec-Dn-3} by $LH$. Then
\begin{align*}
LH&=\sum_{i=\sigma_{k-1}+1}^{\sigma_k}q^{(i-1-\sigma_{k-1})(c+1)}
\sum_{j=-1}^{c-1}\frac{q^{(j+1)\big(c(n-1)+\sigma_k-\sigma_{k-1}-1\big)}}{(q^{-j-1})_{j+1}(q)_{c-j-1}}\\
&=\frac{1-q^{n_k(c+1)}}{1-q^{c+1}}
\sum_{j=0}^{c}\frac{q^{j\big(c(n-1)+\sigma_k-\sigma_{k-1}-1\big)}}{(q^{-j})_{j}(q)_{c-j}}
=\frac{1-q^{n_k(c+1)}}{1-q^{c+1}}\qbinom{nc+n_k-1}{c}.
\end{align*}
The last equation holds by Lemma~\ref{prop-sum} below.
\end{proof}
\begin{lem}\label{prop-sum}
Let $n$ and $t$ be nonnegative integers. Then
\begin{equation}\label{e-sum}
\sum_{j=0}^{t}
\frac{q^{j(n-t)}}
{(q^{-j})_{j}(q)_{t-j}}
=\qbinom{n}{t}.
\end{equation}
\end{lem}
\begin{proof}
We can rewrite the left-hand side of \eqref{e-sum} as
\begin{equation}\label{e-sum2}
\sum_{j=0}^{t}
\frac{(-1)^jq^{j(n-t)+\binom{j+1}{2}}}{(q)_{j}(q)_{t-j}}
=\frac{1}{(q)_t} \sum_{j=0}^{t}q^{\binom{j}{2}}\qbinom{t}{j}(-q^{n-t+1})^j.
\end{equation}
Using the well-known $q$-binomial theorem \cite[Theorem 3.3]{andrew-qbinomial}
\begin{equation}\label{e-qbinom}
(z)_t=\sum_{j=0}^tq^{\binom{j}{2}}\qbinom{t}{j}(-z)^j
\end{equation}
with $z\mapsto q^{n-t+1}$, we have
\[
\frac{1}{(q)_t} \sum_{j=0}^{t}q^{\binom{j}{2}}\qbinom{t}{j}(-q^{n-t+1})^j
=(q^{n-t+1})_t/(q)_t=\qbinom{n}{t}. \qedhere
\]
\end{proof}

\subsection{A family of vanishing coefficients}\label{sec-SP2}

Recall that
\[
L_{n_0,\dots,n_p}(0,0;x)=\prod_{1\leq i<j\leq n}(x_i/x_j)_{c+\varepsilon_{ij}}(qx_j/x_i)_{c+\varepsilon_{ij}}.
\]
By Lemma~\ref{lem-split},
we find a family of vanishing coefficients of $L_{n_0,\dots,n_p}(0,0;x)$ in this subsection.
\begin{lem}\label{lem-vanishing}
Let $n_0$, $h_1,\cdots,h_p$ be integers such that  $2\leq n_0\leq n-1$ and $h_1+\cdots+h_p\leq n_0-1$.
Then
\begin{equation}\label{V0-2}
\CT_{\x}\frac{x_1x_2\cdots x_{n_0}}{\prod_{u=1}^p\prod_{v\in N_u}x_v^{h_u}}\prod_{l=1}^nx_l^{t_l}
\times \prod_{1\leq i<j\leq n}(x_i/x_j)_{c+\varepsilon_{ij}}(qx_j/x_i)_{c+\varepsilon_{ij}}=0,
\end{equation}
where $t_1,t_2,\dots,t_n$ are nonnegative integers such that
\begin{equation}\label{e-tl}
\sum_{l=1}^nt_l=\sum_{u=1}^ph_un_u-n_0.
\end{equation}
\end{lem}
Note that some $h_u$ can be negative integers.
This lemma reduces to \cite[Lemma 3.10]{XZ2023}
for $p=1$.
\begin{proof}
Denote by $\mathfrak{L}$ the left-hand side of \eqref{V0-2}.
If $c=0$ then $\mathfrak{L}$ reduces to
\[
\CT_{\x}\frac{x_1x_2\cdots x_{n_0}}{\prod_{u=1}^p\prod_{v\in N_u}x_v^{h_u}}
\prod_{l=1}^nx_l^{t_l}\times
\prod_{n_0+1\leq i<j\leq n}(x_i/x_j)_{\varepsilon_{ij}}(qx_j/x_i)_{\varepsilon_{ij}},
\]
which is clearly 0 by taking the constant term with respect to $x_i$ for any $i\in \{1,\dots,n_0\}$.
Then we prove $\mathfrak{L}=0$ for $c\geq 1$ below by the induction on $n$.

We first discuss the induction basis.
The $n=3$ and $n_0=2$ case: this forces $p=1$ and $h_1=1$ by $h_1\leq 1=n_0-1$.
In this case, there is no $t_l$ can satisfy \eqref{e-tl}.
The $n=4$ and $n_0=2$ case: this also forces $p=1$ and $h_1=1$, otherwise
there is also no $t_l$ can satisfy \eqref{e-tl}.
Therefore, the initial case satisfies the condition in the lemma should be $n=4, n_0=2, p=1, h_1=1$ and all the $t_l=0$.
That is
\[
\CT_{\x}\frac{x_1x_2}{x_3x_4}
\prod_{1\leq i<j\leq 4}(x_i/x_j)_{c+\chi(i\geq 3)}(qx_j/x_i)_{c+\chi(i\geq 3)}.
\]
The above constant term equals zero by \cite[Example 3.13]{XZ2023}.

We then proceed the induction by assuming that the lemma holds for $n\mapsto n-1$.

It is easy to see that
\begin{align*}
\mathfrak{L}&=\CT_{\x,y}\frac{x_1x_2\cdots x_{n_0}}{\prod_{u=1}^p\prod_{v\in N_u}x_v^{h_u}}
\prod_{l=1}^nx_l^{t_l} \frac{\prod_{1\leq i<j\leq n}(x_i/x_j)_{c+\varepsilon_{ij}}(qx_j/x_i)_{c+\varepsilon_{ij}}}
{\prod_{i=0}^{k-1}\prod_{l\in N_i}(y/x_l)_{c}
\prod_{l\in N_k}(q^{-1}y/x_l)_{c+1}
\prod_{i=k+1}^{p}\prod_{l\in N_i}(q^{-1}y/x_l)_{c}}\\
&=\CT_{\x,y}\frac{x_1x_2\cdots x_{n_0}}{\prod_{u=1}^p\prod_{v\in N_u}x_v^{h_u}}
\prod_{l=1}^nx_l^{t_l}\times S\big((n_0,\dots,n_p),c;x,y\big),
\end{align*}
where all forms of $1/(1-q^zy/x_l)$ satisfy \eqref{e-explain} and $S$ is defined in \eqref{def-L}.
Using the splitting formula \eqref{e-sp-main} for $S$ yields
\begin{multline*}
\mathfrak{L}=\CT_{\x,y}\frac{x_1x_2\cdots x_{n_0}}{\prod_{u=1}^p\prod_{v\in N_u}x_v^{h_u}}
\prod_{l=1}^nx_l^{t_l}\\
\times\bigg(
\sum_{t=0}^{k-1}\sum_{i\in N_t}\sum_{j=0}^{c-1}\frac{A_{ij}^{(t)}}{1-q^jy/x_i}
+\sum_{i\in N_k}\sum_{j=-1}^{c-1}\frac{A^{(k)}_{ij}}{1-q^jy/x_i}
+\sum_{t=k+1}^{p}\sum_{i\in N_t}\sum_{j=-1}^{c-2}\frac{A_{ij}^{(t)}}{1-q^jy/x_i}
\bigg).
\end{multline*}
Extracting the constant term with respect to $y$ gives
\begin{equation}\label{V1}
\mathfrak{L}=\CT_{\x,y}\frac{x_1x_2\cdots x_{n_0}}{\prod_{u=1}^p\prod_{v\in N_u}x_v^{h_u}}
\prod_{l=1}^nx_l^{t_l}
\times\bigg(
\sum_{t=0}^{k-1}\sum_{i\in N_t}\sum_{j=0}^{c-1}A_{ij}^{(t)}
+\sum_{i\in N_k}\sum_{j=-1}^{c-1}A^{(k)}_{ij}
+\sum_{t=k+1}^{p}\sum_{i\in N_t}\sum_{j=-1}^{c-2}A_{ij}^{(t)}
\bigg).
\end{equation}
We use $\hat{x}_i$ to denote the omission of $x_i$.
Let
\begin{subequations}\label{C}
\begin{equation}\label{C0}
C_0:=\CT_{\x}\frac{x_1\cdots \hat{x_i}\cdots x_{n_0}}{\prod_{u=1}^p\prod_{v\in N_u}x_v^{h_u}}
\times x_i^{t_i+1}\prod_{\substack{l=1\\l\neq i}}^nx_l^{t_l}\times A^{(0)}_{ij}
\end{equation}
for $i=1,\dots,n_0$ and $j=0,\dots,c-1$;
\begin{equation}\label{C1}
C_1:=\CT_{\x}\frac{x_1\cdots x_{n_0}}{\prod_{u=1}^p\prod_{\substack{v\in N_u\\v\neq i}}x_v^{h_u}}
\times x_i^{t_i-h_t}\prod_{\substack{l=1\\l\neq i}}^nx_l^{t_l}\times A^{(t)}_{ij}
\end{equation}
for $i\in N_1\cup \cdots \cup N_{k-1}$ and $j=0,\dots,c-1$;
\begin{equation}\label{C2}
C_2:=\CT_{\x}\frac{x_1\cdots x_{n_0}}{\prod_{u=1}^p\prod_{\substack{v\in N_u\\v\neq i}}x_v^{h_u}}
\times x_i^{t_i-h_k}\prod_{\substack{l=1\\l\neq i}}^nx_l^{t_l}\times A^{(k)}_{ij}
\end{equation}
for $i\in N_{k}$ and $j=-1,\dots,c-1$;
\begin{equation}\label{C3}
C_3:=\CT_{\x}\frac{x_1\cdots x_{n_0}}{\prod_{u=1}^p\prod_{\substack{v\in N_u\\v\neq i}}x_v^{h_u}}
\times x_i^{t_i-h_t}\prod_{\substack{l=1\\l\neq i}}^nx_l^{t_l}\times A^{(t)}_{ij}
\end{equation}
\end{subequations}
for $i\in N_{k+1}\cup \cdots \cup N_p$ and $j=-1,\dots,c-2$.
In $C_1$ and $C_3$, for a given $i$ there exists a unique $t\in \{0,\dots,p\}$ such that $i\in N_t$.
Hence, the parameter $t$ is determined by $i$ in each $A_{ij}^{(t)}$.
Also note that we suppress the parameters in the $C_l$.
By \eqref{V1}, it is clear that $\mathfrak{L}$ is a finite sum of
the forms of $C_l$. We show that all the $C_l=0$.

Notice that $A_{ij}^{(0)}$ is a polynomial in $x_i^{-1}$ with degree at least $n_k$.
If $t_i+1<n_k$, then $C_0=0$ since the Laurent polynomial in \eqref{C0} is actually
a polynomial in $x_i^{-1}$ with no constant term.
For $t_i+1\geq n_k$,
by the expression for $A^{(0)}_{ij}$ in \eqref{A-0} and taking the constant term
with respect to $x_i$, the constant term $C_0$ can be written as a finite sum of the form
\begin{align}\label{C0-1}
&c\times\CT_{x^{(i)}}\frac{x_1\cdots \hat{x_i}\cdots x_{n_0}}
{\prod_{\substack{u=1\\u\neq k}}^p\prod_{v\in N_u}x_v^{h_u}\prod_{v\in N_k}x_v^{h_k-1}}
\prod_{\substack{l=1\\l\neq i}}^nx_l^{t'_l}
\times \prod_{\substack{1\leq u<v\leq n\\u,v\neq i}}(x_u/x_v)_{c+\varepsilon_{uv}}(qx_v/x_u)_{c+\varepsilon_{uv}}\nonumber\\
&\quad =c\times\CT_{x^{(i)}}\frac{\prod_{v\in N_0'}x_v}
{\prod_{u=1}^p\prod_{v\in N'_u}x_v^{h'_u}}
\prod_{\substack{l=1\\l\neq i}}^nx_l^{t'_l}L_{n_0-1,n_1,\dots,n_p}(0,0;x^{(i)}).
\end{align}
Here $N'_0=N_0\setminus \{i\}$, $N'_i=N_i$ for $i=1,\dots,p$,
$h'_k=h_k-1$, $h'_v=h_v$ for $i\in \{1,\dots,p\}\setminus \{k\}$,
$c\in \mathbb{R}(q)$, the $t'_l$ are nonnegative integers such that
$\sum_{\substack{l=1\\l\neq i}}^nt'_l=\sum_{u=1}^ph'_un_u-n_0+1$, and $x^{(i)}=(x_1,\dots,\hat{x_i},\dots,x_n)$.
Since $\sum_{u=1}^ph'_u=\sum_{u=1}^ph_u-1\leq n_0-2$,
the constant term in \eqref{C0-1} equals 0 by the induction hypothesis.
Then we have $C_0=0$.

It is known that $A^{(t)}_{ij}$ is a polynomial in $x_i^{-1}$ with degree at least $n_k-n_t+1$
for $t=1,\dots,k-1$.
If $t_i-h_t<n_k-n_t+1$ then $C_1=0$ since the Laurent polynomial in \eqref{C1}
is in fact a polynomial in $x_i^{-1}$ with no constant term.
If $t_i-h_t\geq n_k-n_t+1$,  by the expression for $A^{(t)}_{ij}$ in \eqref{A-1} and taking the constant term with respect to $x_i$,
the constant term $C_1$ can be written as a finite sum of the form
\begin{align}\label{C1-1}
&c\times\frac{x_1\cdots \cdots x_{n_0}}
{\prod_{\substack{u=1\\u\neq t,k}}^p\prod_{v\in N_u}x_v^{h_u}
\prod_{\substack{v\in N_t\\v\neq i}}x_v^{h_t+1}
\prod_{v\in N_k}x_v^{h_k-1}}
\prod_{\substack{l=1\\l\neq i}}^nx_l^{t'_l}
\times \prod_{\substack{1\leq u<v\leq n\\u,v\neq i}}(x_u/x_v)_{c+\varepsilon_{uv}}(qx_v/x_u)_{c+\varepsilon_{uv}}\nonumber\\
&\quad =c\times\CT_{x^{(i)}}\frac{\prod_{v\in N_0'}x_v}
{\prod_{u=1}^p\prod_{v\in N'_u}x_v^{h'_u}}
\prod_{\substack{l=1\\l\neq i}}^nx_l^{t'_l}L_{n_0,n_1,\dots,n_{t-1},n_t-1,n_{t+1},\dots,n_p}(0,0;x^{(i)}).
\end{align}
Here $N'_0=N_0$, $N'_i=N_i$ for $i\in \{1,\dots,p\}\setminus \{t\}$,
$N'_t=N_t\setminus \{i\}$, $n'_i=|N'_i|$,
$h'_k=h_k-1$, $h'_t=h_t+1$, $h'_v=h_v$ for $i\in \{1,\dots,p\}\setminus \{t,k\}$,
$c\in \mathbb{R}(q)$ and the $t'_l$ are nonnegative integers such that
$\sum_{\substack{l=1\\l\neq i}}^nt'_l=\sum_{u=1}^ph'_un'_u-n_0$.
The sum
\[
\sum_{u=1}^ph'_{u}=\sum_{\substack{u=1\\u\neq t,k}}^ph_u+(h_t+1)+(h_k-1)=\sum_{u=1}^ph_u\leq n_0-1.
\]
Then the constant term in \eqref{C1-1} equals 0 by the induction hypothesis if $n_0<n-1$.
For $n_0=n-1$, the constant term $C_1$ does not exist (or we can say $C_1=0$ for this case).
Hence, we have $C_1=0$.

Since $A^{(k)}_{ij}$ is a polynomial in $x_i^{-1}$, the constant term \eqref{C2}
vanishes if $t_i<h_k$. If $t_i\geq h_k$, using the expression for $A^{(k)}_{ij}$
in \eqref{A-2} and taking the constant term with respect to $x_i$,
the constant term $C_2$ can be written as a finite sum of the form
\begin{align}\label{C2-1}
&c\times\frac{x_1\cdots \cdots x_{n_0}}
{\prod_{\substack{u=1\\u\neq k}}^p\prod_{v\in N_u}x_v^{h_u}
\prod_{\substack{v\in N_k\\v\neq i}}x_v^{h_k}}
\prod_{\substack{l=1\\l\neq i}}^nx_l^{t'_l}
\times \prod_{\substack{1\leq u<v\leq n\\u,v\neq i}}(x_u/x_v)_{c+\varepsilon_{uv}}(qx_v/x_u)_{c+\varepsilon_{uv}}\nonumber\\
&\quad =c\times\CT_{x^{(i)}}\frac{\prod_{v\in N_0'}x_v}
{\prod_{u=1}^p\prod_{v\in N'_u}x_v^{h'_u}}
\prod_{\substack{l=1\\l\neq i}}^nx_l^{t'_l}L_{n_0,n_1,\dots,n_{k-1},n_k-1,n_{k+1},\dots,n_p}(0,0;x^{(i)}).
\end{align}
Here $N'_i=N_i$ for $i\in \{0,\dots,p\}\setminus \{k\}$,
$N'_k=N_k\setminus \{i\}$, $n'_i=|N'_i|$,
$h'_v=h_v$ for $i\in \{1,\dots,p\}$,
$c\in \mathbb{R}(q)$ and the $t'_l$ are nonnegative integers such that
$\sum_{\substack{l=1\\l\neq i}}^nt'_l=\sum_{u=1}^ph'_un'_u-n_0$.
The sum
\[
\sum_{u=1}^ph'_{u}=\sum_{u=1}^ph_u\leq n_0-1.
\]
Then the constant term \eqref{C2-1} equals 0 by the induction hypothesis if $n_0<n-1$.
For $n_0=n-1$, the constant term \eqref{C2-1} also vanishes since it can not be homogeneous.
Then we have $C_2=0$.

The constant term $C_3=0$ by the similar argument as that of $C_1$.

In conclusion, we obtain that all the forms of $C_0,C_1,C_2,C_3$ defined in \eqref{C}
are zeros. Then $\mathfrak{L}=0$, completing the proof.
\end{proof}

\section{The Gessel--Xin method to evaluate constant terms}\label{sec-GX}

In this subsection, we briefly introduce the Gessel--Xin method to evaluate
constant terms of rational functions.

The method is based on the theory of iterated Laurent series. Xin constructed
the field of iterated Laurent series in his Ph.D thesis \cite{xinresidue}.
In the following of this paper, we work in the field of iterated
Laurent series $\mathbb{C}\langle\!\langle x_n, x_{n-1},\dots,x_0\rangle\!\rangle
=\mathbb{C}(\!(x_n)\!)(\!(x_{n-1})\!)\cdots (\!(x_0)\!)$.
Elements of $\mathbb{C}\langle\!\langle x_n,x_{n-1},\dots,x_0\rangle\!\rangle$
are regarded first as Laurent series in $x_0$, then as
Laurent series in $x_1$, and so on.
The application of this theory to constant terms first appeared
in Gessel and Xin's~\cite{GX} Laurent series proof of the Zeilberger--Bressoud
$q$-Dyson constant term identity \cite{andrews1975,zeil-bres1985}.
For more applications to constant terms, the reader could refer to \cite{LXZ, xiniterate, XZ, XZ2023, Zhou, Zhou23}.
 The most applicable fact in what is to follow is that the field $\mathbb{C}(x_0,\dots,x_n)$ of
rational functions in the variables $x_0,\dots,x_n$ with coefficients in $\mathbb{C}$
forms a subfield of $\mathbb{C}\langle\!\langle x_n, x_{n-1},\dots,x_0\rangle\!\rangle$, so that every rational function can be identified as its unique Laurent series expansion.

In our approach, the next expansion of $1/(1-cx_i/x_j)$ for $c\in \mathbb{C}\setminus \{0\}$
is fundamental:
\[
\frac{1}{1-c x_i/x_j}=
\begin{cases} \displaystyle \sum_{l\geq 0} c^l (x_i/x_j)^l
& \text{if $i<j$}, \\[5mm]
\displaystyle -\sum_{l<0} c^l (x_i/x_j)^l
& \text{if $i>j$}.
\end{cases}
\]
Thus,
\begin{equation}
\label{e-ct}
\CT_{x_i} \frac{1}{1-c x_i/x_j} =
\begin{cases}
    1 & \text{if $i<j$}, \\
    0 & \text{if $i>j$}, \\
\end{cases}
\end{equation}
where, for $f\in \mathbb{C}\langle\!\langle x_n, x_{n-1},\dots,x_0\rangle\!\rangle$,
the operator $\displaystyle\CT_{x_i}$ also means to take the constant term
with respect to $x_i$.
The constant term operators defined in $\mathbb{C}\langle\!\langle x_n,\dots,x_0\rangle\!\rangle$ have the property of commutativity:
\[
\CT_{x_i} \CT _{x_j} f = \CT_{x_j} \CT_{x_i} f
\]
for any $i$ and $j$.
This means that we are taking constant term in a set of variables.
Then $\CT\limits_x$ is well-defined in this field.

The following lemma has appeared previously in~\cite{GX}.
It is a basic tool for
extracting constant terms from rational functions.
\begin{lem}\label{lem-almostprop}
For a positive integer $m$, let $p(x_k)$ be a Laurent polynomial
in $x_k$ of degree at most $m-1$ with coefficients in
$K\langle\!\langle x_n,\dots,x_{k-1},x_{k+1},\dots,x_0\rangle\!\rangle$.
Let $0\leq i_1\leq\dots\leq i_m\leq n$ such that all $i_r\neq k$,
and define
\begin{equation}\label{e-defF}
f=\frac{p(x_k)}{\prod_{r=1}^m (1-c_r x_k/x_{i_r})},
\end{equation}
where $c_1,\dots,c_m\in \mathbb{C}\setminus \{0\}$ such that $c_r\neq c_s$ if $x_{i_r}=x_{i_s}$.
Then
\begin{equation}\label{e-almostprop}
\CT_{x_k} f=\sum_{\substack{r=1 \\[1pt] i_r>k}}^m
\big(f\,(1-c_rx_k/x_{i_r})\big)\Big|_{x_k=c_r^{-1}x_{i_r}}.
\end{equation}
\end{lem}
Note that we extended Lemma~\ref{lem-almostprop} to \cite[Lemma~3.1]{LXZ},
in which $p(x_k)$ is a Laurent polynomial in $x_k$ of degree at most $m$.
In this paper, we need to deal with the cases when the degree in $x_k$ of $p(x_k)$ is larger than $m$. This technique first appeared in \cite{XZ2023}.

If $f$ is a rational function of the form
\eqref{e-defF} with respect to $x_0$. Then we can apply Lemma~\ref{lem-almostprop}
to $f$ with respect to $x_0$ and obtain \eqref{e-almostprop} for $k=0$.
We can write \eqref{e-almostprop} as a finite sum of the form
\[
\CT_{x_0} f=\sum_{1\leq u_1\leq n} f_{(u_1)}.
\]
Note that $f_{(u_1)}$ may be a sum. It is still a multi-variable rational function without the variable $x_0$.
If furthermore, one of the $f_{(u_1)}$ is still of the form \eqref{e-defF} with respect to $x_{u_1}$, then we can apply Lemma~\ref{lem-almostprop} again to the $f_{(u_1)}$ and obtain
\[
\CT_{x_{u_1}} f_{(u_1)}=\sum_{1\leq u_1<u_2\leq n} f_{(u_1,u_2)}.
\]
We can further apply Lemma~\ref{lem-almostprop} to each $f_{(u_1,u_2)}$ with respect to $x_{u_2}$ if applicable, and get a sum.
Continue this operation until Lemma~\ref{lem-almostprop} is no longer valid to every summand. Finally we write
\begin{equation}\label{hugesum}
\CT_{\x}f=\sum_{s\in T}\sum_{1\leq u_1<\cdots<u_s\leq n}
\CT_{\x}f_{(u_1,\dots,u_s)},
\end{equation}
where $T$ is a subset of $\{1,\dots,n\}$. Of course, here $f_{(u_1,\dots,u_s)}$ may
also be a sum.
We call this operation the Gessel--Xin operation to the rational function $f$, since it first appeared in Gessel and Xin's paper \cite{GX}.
Though we turn $\CT\limits_{\x}f$ into a huge sum \eqref{hugesum},
we can determine that all the summand in
\eqref{hugesum} are zeros under certain conditions.
Then we can arrive at a conclusion that
\[
\CT_{\x} f=0.
\]
This is the main procedure of the Gessel--Xin method.

\section{The roots}\label{sec-roots}

In this section, we determine all the roots of $D_n(a)$ and prove Lemma~\ref{lem-roots}.

\subsection{Preliminaries}

For a positive integer $s$, let $\mathfrak{S}_s$ be the set of all the permutations of $\{1,2,\dots,s\}$.
For positive integers $r_0,r_1,\dots,r_p$ such that $r_0+\cdots+r_p=s$,
denote $r=(r_0,\dots,r_p)$ and
\begin{equation}\label{def-Ri}
R_0=\{1,\dots,r_0\},\qquad R_i=\{r_0+\cdots+r_{i-1}+1,\dots,r_0+\cdots+r_i\}
\quad \text{for} \quad i=1,\dots,p.
\end{equation}
For $w\in \mathfrak{S}_s$,
define
\begin{equation}\label{Nwr}
N_{w,r}:=\sum_{j=1}^se_j(w,r),
\end{equation}
where
\begin{equation}\label{defi-ej}
e_j=e_j(w,r):=u_j(w)+\chi\big(\{w(j),w(j-1)\}\subseteq R_i \ \text{for some $i=1,\dots,p$}\big).
\end{equation}
Here $w(0):=0$ and
\begin{equation}\label{def-uj}
u_j(w)=\begin{cases}
1 &\text{if $w(j-1)<w(j)$,}\\
0 &\text{otherwise.}
\end{cases}
\end{equation}
For example, if $s=10$, $r=(r_0,r_1,r_2)=(3,3,4)$, and $w=(9,10,3,5,6,8,4,2,7,1)$,
then $R_0=\{1,2,3\}$, $R_1=\{4,5,6\}$, $R_2=\{7,8,9,10\}$ and $N_{w,r}=8$.
Let
\begin{equation}\label{e-weight}
0=w(0)\mathop{\longrightarrow}^{e_1}w(1)\mathop{\longrightarrow}^{e_2} \cdots\mathop{\longrightarrow}^{e_{s-1}} w(s-1)\mathop{\longrightarrow}^{e_s}w(s)
\end{equation}
be a weighted directed path on $w\in \mathfrak{S}_s$.
Then we can view $N_{w,r}$ as the sum of the weights on this directed path.
Thus, the above example can be represented graphically as follows:
\[
0\mathop{\longrightarrow}^{1}9\mathop{\longrightarrow}^{2}10\mathop{\longrightarrow}^{0}
3\mathop{\longrightarrow}^{1}5\mathop{\longrightarrow}^{2}
6\mathop{\longrightarrow}^{1}8\mathop{\longrightarrow}^{0}
4\mathop{\longrightarrow}^{0}2\mathop{\longrightarrow}^{1}
7\mathop{\longrightarrow}^{0}1.
\]
Note that $e_j$ is determined by $w(j), w(j-1)$ and $r$,
and $e_1=1$. We call $w(j-1)\mathop{\longrightarrow}\limits^{e_j}w(j)$ an ascending arc if $w(j-1)<w(j)$.

It is not hard to obtain the lower bound for $N_{w,r}$.
\begin{prop}\label{prop-min}
For a positive integer $s$, let $w\in \mathfrak{S}_s$, $r=(r_0,\dots,r_p)$ be a vector of  positive
integers such that $|r|=s$,
and $N_{w,r}$ be defined in \eqref{Nwr}.
Then, for all $w\in \mathfrak{S}_s$ and any $i\in \{1,\dots,s\}$
\begin{equation}\label{e-lowerbound}
N_{w,r}=\sum_{j=1}^se_j(w,r)\geq \max\{r_1,\dots,r_p\},
\end{equation}
and
\begin{equation}\label{lowerbound1}
\sum_{\substack{j=1\\j\neq i}}^se_j(w,r)\geq \max\{r_1,\dots,r_p\}-1.
\end{equation}
\end{prop}
\begin{proof}
Recall the definition of  $R_i$ in \eqref{def-Ri}.
For a given $l\in \{1,2,\dots,p\}$ and any two elements $w(i_1), w(i_2)\in R_l$,
we show that the sum of weights that start from $w(i_1)$ and end at $w(i_2)$ must be no less than 1. Denote this sum by $h$.
More explicitly, either $w(i_1)\mathop{\longrightarrow}\limits^{e_{i_2}}w(i_2)$ or
\[
w(i_1)  \mathop{\longrightarrow}^{e_j}w(j)\overbrace{\longrightarrow\cdots \mathop{\longrightarrow}}^{h_1}w(k) \mathop{\longrightarrow}^{e_{i_2}}w(i_2),
\]
then $h=e_{i_2}\geq 1$ for the first case and
\begin{equation}\label{e-proplowerbound-1}
h=e_j+e_{i_2}+h_1\geq 1
\end{equation}
for the second case.
The first case holds apparently since both $w(i_1), w(i_2)\in R_l$.
We show the second case by discussing:
(i) $w(i_1)>w(j), w(k)<w(i_2)$;
(ii) $w(i_1)<w(j), w(k)>w(i_2)$;
(iii) $w(i_1)<w(j), w(k)<w(i_2)$;
(iv) $w(i_1)>w(j), w(k)>w(i_2)$.
For (i)--(iii), it is easy to see that $h\geq e_j+e_{i_2} \geq 1$ since there exists at least one ascending arc
between two consecutive elements.
The last case $w(i_1)>w(j), w(k)>w(i_2)$. 
If one of $w(j)$ and $w(k)$ belongs to $R_l$, then either $e_j=1$ or $e_{i_2}=1$ and $h\geq 1$.
Thus, we assume $w(j),w(k)\notin R_l$ below. In this case $w(j)\neq w(k)$, otherwise $w(i_1)>w(j)$ and $w(k)>w(i_2)$ can not hold both for
$w(j)=w(k)\notin R_l$ .
Since $w(j)\neq w(k)$, $w(i_1)>w(j)$, $w(k)>w(i_2)$ and $w(j),w(k)\notin R_l$, we conclude that $w(j)<w(i_1)<w(k)$.
Using the fact that $w(j)<w(k)$, there must exist at least one ascending arc between $w(j)$ and $w(k)$. Hence $h_1\geq 1$.
It follows that $h\geq 1$ by \eqref{e-proplowerbound-1}.

Let $i_1$ and $i_2$ be any two consecutive elements in $R_l$.
Using the above fact that the sum of weights that start from $w(i_1)$ and end at $w(i_2)$ must be no less than 1, together with $e_1=1$, we have
\[
N_{w,r}\geq r_l=|R_l|
\]
for any given $l\in \{1,\dots,p\}$.
Then $N_{w,r}\geq \max\{r_1,\dots,r_p\}$.

By the discuss above, it is clear that for any $i\in \{1,\dots,s\}$
\[
\sum_{\substack{j=1\\j\neq i}}^se_j(w,r)\geq \max\{r_1-1,\dots,r_p-1\}
=\max\{r_1,\dots,r_p\}-1.\qedhere
\]
\end{proof}
Note that the lower bound in \eqref{e-lowerbound} can be attained.
To determine the minimum $N_{w,r}$, the approach involves separating the elements of $R_i$ as much as possible and then sorting these elements in descending order, from largest to smallest.
To demonstrate the corresponding $w\in \mathfrak{S}_s$,
we name the elements of $R_i$ (in an increasing order) as $a_{i,1},\dots,a_{i,r_i}$ respectively.
Note that $|R_i|=r_i$.
Arrange all the $a_{i,j}$ as follows
\[
w^{0}=\big(a_{p,r_p}, a_{p-1,r_{p-1}},\dots,a_{0,r_0},
a_{p,r_{p}-1},a_{p-1,r_{p-1}-1},\dots,a_{0,r_0-1},\dots
\big),
\]
where $a_{p,r_p},\dots,a_{0,r_0}$
are the largest elements in $R_p,\dots,R_0$ respectively,
and $a_{p,r_{p}-1},\dots,a_{0,r_0-1}$
are the second largest elements in $R_p,\dots,R_0$ respectively, and so on.
It is not hard to see that $N_{w^0,r}=\max\{r_1,\dots,r_p\}$.
For example, if $s=10$, $r=(3,3,4)$, $R_0=\{1,2,3\}$, $R_1=\{4,5,6\}$ and $R_2=\{7,8,9,10\}$,
then $w^0$ in $\mathfrak{S}_{10}$ is
$(10,6,3,9,5,2,8,4,1,7)$ and $N_{w^0,r}=4=\max\{r_1,r_2\}$.

We present the key lemma, which plays a pivotal role in determining the roots of $D_n(a)$.
\begin{lem}\label{lem-key}
For a positive integer $s$, let $b,c,t$ and $k_1,\dots,k_s$ be nonnegative integers such that
$1\leq k_{i}\leq (s-1)c+b+t$ for $1\leq i\leq s$.
Then at least one of the following holds:
\begin{enumerate}
\item $1\leq k_i\leq b$ for some $i$ with $1\leq i\leq s$;
\item $-c\leq k_i-k_j\leq c-1$ for some $(i,j)$ such that $1\leq i<j\leq s$
and $\{i,j\}\nsubseteq R_l$ for all $l=1,2,\dots,p$;
\item $-c-1\leq k_{i}-k_{j}\leq c$ for some $(i,j)$ such that $1\leq i<j\leq s$
and $\{i,j\}\subseteq R_l$ for some $l\in \{1,2,\dots,p\}$;
\item there exists a permutation $w\in\mathfrak{S}_s$ and nonnegative integers $d_1,\dots,d_s$
such that
\begin{subequations}\label{e-k1-k}
\begin{equation}\label{e-k1-1-k}
k_{w(1)}=b+d_1,
\end{equation}
and
\begin{equation}\label{e-k1-2-k}
k_{w(j)}-k_{w(j-1)}=c+\chi\big(\{w(j),w(j-1)\}\subseteq R_i \ \text{for some $i=1,\dots,p$}\big)+d_j
 \quad \text{for $2\leq j\leq s$.}
\end{equation}
\end{subequations}
Here the $d_j$ satisfy
\begin{equation}\label{tt}
\max\{r_1,\dots,r_p\}\leq \sum_{j=1}^{s}\Big(\chi\big(\{w(j),w(j-1)\}\subseteq R_i \ \text{for some $i=1,\dots,p$}\big)+d_j\Big)\leq t,
\end{equation}
$w(0):=0$, and $d_j>0$ if $w(j-1)<w(j)$ for $1\leq j\leq s$.
\end{enumerate}
\end{lem}
Note that the $p=0$ case reduces to \cite[Lemma 8.1]{Zhou23},
the $p=1$ case reduces to \cite[Lemma 6.4]{XZ2023}.
Also notice that $d_1$ is always a positive integer, and if $t=0$ then \eqref{tt} can not hold. Hence, (4) can not occur if $t=0$.
\begin{proof}
We prove the lemma by showing that if (1)--(3) fail then (4) must hold.

Assume that (1)--(3) are all false.
Then we construct a weighted tournament $T$ on a complete graph
on $s$ vertices, labelled $1,\dots,s$, as follows.
For the edge $(i,j)$ with $1\leq i<j\leq s$ and $\{i,j\}\nsubseteq R_l$ for any $l\in \{1,2,\dots,p\}$,
we draw an arrow from $j$ to $i$ and attach a weight $c$ if $k_i-k_j\ge c$.
If, on the other hand, $k_i-k_j\le -c-1$ then we draw an arrow
from $i$ to $j$ and attach the weight $c+1$.
Similarly, for the edge $(i,j)$ with $1\leq i<j\leq s$
and $\{i,j\}\subseteq R_l$ for some $l\in \{1,2,\dots,p\}$,
we draw an arrow from $j$ to $i$ and attach a weight $c+1$ if $k_i-k_j\ge c+1$,
and we draw an arrow from $i$ to $j$ and attach the weight $c+2$
if $k_i-k_j\le -c-2$.
Note that the weight of each edge of a tournament is nonnegative since $c$ is a nonnegative integer.

We call a directed edge from $i$ to $j$ ascending if $i<j$.
It is immediate from our construction that
(i) the weight of the edge $i\to j$ is less than or equal $k_j-k_i$, and
(ii) the weight of an ascending edge is positive.

We will use (i) and (ii) to show that any of the above-constructed
tournaments is acyclic and hence transitive.
As a consequence of (i), the weight of a directed path from $i$ to $j$ in $T$,
defined as the sum of the weights of its edges, is at most $k_j-k_i$.
Proceeding by contradiction, assume that $T$ contains a cycle $C$.
By the above, the weight of $C$ must be non-positive, and hence $0$.
Since $C$ must have at least one ascending edge, which by (ii) has a positive
weight, the weight of $C$ is positive, a contradiction.

Since each $T$ is transitive, there is exactly one directed Hamilton path $P$ in $T$,
corresponding to a total order of the vertices.
Assume $P$ is given by
\[
P=w(1)\rightarrow w(2)\rightarrow\cdots\rightarrow w(s-1)\rightarrow w(s),
\]
where we have suppressed the edge weights.
Then
\[
k_{w(s)}-k_{w(1)}\ge (s-1)c+\sum_{j=2}^s \chi\big(\{w(j),w(j-1)\}\subseteq R_i \ \text{for some $i=1,\dots,p$}\big),
\]
and thus
\begin{align}\label{e-contradiction}
k_{w(s)}&\ge k_{w(1)}+(s-1)c+\sum_{j=2}^s \chi\big(\{w(j),w(j-1)\}\subseteq R_i \ \text{for some $i=1,\dots,p$}\big)\nonumber\\
    & \ge b+1+(s-1)c+\sum_{j=2}^s \chi\big(\{w(j),w(j-1)\}\subseteq R_i \ \text{for some $i=1,\dots,p$}\big).
\end{align}
Together with the assumption that $k_{w(s)}\leq (s-1)c+b+t$
this implies that $P$ has at most $t-1$ ascending edges.
Let $d_1,\dots,d_s$ be nonnegative integers such that \eqref{e-k1-k} holds.
Since (1) does not hold and by \eqref{e-k1-1-k}, $k_{w(1)}=b+d_1\geq b+1$, so that $d_1\geq 1$.
For $2\leq j\leq s$, if $w(j-1)\to w(j)$ is an ascending edge, then $d_j$ is a positive integer.
That is, for $2\leq j\leq s$ if $w(j-1)<w(j)$ then $d_j\geq 1$.
Set $k_0:=0$.
Since
\begin{multline*}
\sum_{j=1}^s (k_{w(j)}-k_{w(j-1)})=k_{w(s)}
=b+(s-1)c+\sum_{j=1}^s \big(d_j+\chi\big(\{w(j),w(j-1)\}\subseteq R_i \ \text{for some $i=1,\dots,p$}\big) \\
\leq b+(s-1)c+t,
\end{multline*}
we have
\[
\sum_{j=1}^s \big(d_j+\chi\big(\{w(j),w(j-1)\}\subseteq R_i \ \text{for some $i=1,\dots,p$}\big)\big)\leq t.
\]
Since $d_j\geq u_j$ by the definition of the $u_j$ in \eqref{def-uj},
\[
\sum_{j=1}^s \big(d_j+\chi\big(\{w(j),w(j-1)\}\subseteq R_i \ \text{for some $i=1,\dots,p$}\big)\big)
\geq N_{w,r},
\]
where $N_{w,r}$ is defined in \eqref{Nwr}.
Together with Proposition~\ref{prop-min} yields
\[
\sum_{j=1}^s \big(d_j+\chi\big(\{w(j),w(j-1)\}\subseteq R_i \ \text{for some $i=1,\dots,p$}\big)\big)
\geq \max\{r_1,\dots,r_p\}.
\]
This completes the proof of the assertion that (4) must hold if (1)--(3) fail.
\end{proof}

By Lemma~\ref{lem-key}, we have an essential result to determine the roots
of $D_n(a)$.
\begin{cor}\label{cor-3}
For a positive integer $s$, let $b,c,k_1,\dots,k_s$
be nonnegative integers and $r_0,\dots,r_p,n_0,\dots,n_p$ be positive integers such that
$1\leq k_{i}\leq (s-1)c+b+t_s$ for $1\leq i\leq s\leq n$,
$n_0+\cdots+n_p=n$, $r_0+\cdots+r_p=s$, and
$r_i\leq \min \{s,n_i\}$ for $i=0,\dots,p$.
Take
\begin{equation}\label{def-t}
t_{s}=\begin{cases}
0  &\text{if $1\leq s\leq n-\sum_{i=1}^p n_i+p$},\\
\big\lfloor\frac{s-n_0-1}{p} \big\rfloor  &\text{if $n-\sum_{i=1}^p n_i+p+1\leq s
\leq n-\sum_{i=2}^pn_{w_i}+(p-1)n_{w_1}$},\\
\big\lfloor\frac{s-n_0-n_{w_1}-1}{p-1} \big\rfloor 
&\text{if $n-\sum_{i=2}^pn_{w_i}+(p-1)n_{w_1}+1\leq s
\leq n-\sum_{i=3}^pn_{w_i}+(p-2)n_{w_2}$},\\
 \vdots\\
s-n_0-n_{w_1}-\cdots-n_{w_{p-1}}-1  &\text{if
$n-n_{w_p}+n_{w_{p-1}}+1\leq s
\leq n$,}
\end{cases}
\end{equation}
where $(n_{w_1},\dots,n_{w_p})$ is a permutation of $(n_1,\dots,n_p)$ in an increasing order.
In particular, if $p=0$, then $t_s=0$ for $1\leq s\leq n$; if $p=1$, then $t_s=0$ for $1\leq s\leq n_0+1$ and $t_s=s-n_0-1$ for $n_0+2\leq s\leq n$.
Denote by $R_0=\{1,\dots,r_0\}$, $R_i=\{r_0+\cdots+r_{i-1}+1,\dots,r_0+\cdots+r_i\}$
for $i=1,\dots,p$.
Then, at least one of the following cases holds:
\begin{itemize}
\item[(i)] $1\leq k_i\leq b$ for some $i$ with $1\leq i\leq s$;
\item[(ii)] $-c\leq k_i-k_j\leq c-1$ for some $(i,j)$ such that $1\leq i<j\leq s$ and
$\{i,j\}\nsubseteq R_l$ for all $l=1,\dots,p$;
\item[(iii)] $-c-1\leq k_i-k_j\leq c$ for some $(i,j)$ such that $1\leq i<j\leq s$
and $\{i,j\}\subseteq R_l$ for some $l\in \{1,\dots,p\}$.
\end{itemize}
\end{cor}
For the $t_s>0$ cases in \eqref{def-t}, we can summarize the formulas
for $t_s$ according to the rows $j=1,\dots,p$:
\begin{equation}\label{eq-tt}
t_s=\Big\lfloor\frac{s-n_0-\sum_{i=1}^{j-1}n_{w_i}-1}{p-j+1} \Big\rfloor
\end{equation}
for $n-\sum_{i=j}^{p}n_{w_i}+(p-j+1)n_{w_{j-1}}+1\leq s\leq
n-\sum_{i=j+1}^{p}n_{w_i}+(p-j)n_{w_j}$.
Here we take $n_{w_0}=1$.
Note that there are $p+1$ rows in \eqref{def-t}. We identify the first row as the row 0.
The expression for $(s-1)c+b+t_s$ in the lemma is actually the upper bound of $R_i^j$ in \eqref{def-Rj} with $i=s-1$.
\begin{proof}
Compare with Lemma~\ref{lem-key},
it suffices to prove that Case (4) of Lemma~\ref{lem-key}
can not occur under the conditions of this corollary.
We complete this by showing that \eqref{tt} can not hold with $t=t_s$.
For the $t_s=0$ case, \eqref{tt} can not hold since $d_1\geq 1$.
For the $t_s>0$ cases, we show that
\begin{equation}\label{e-inequality}
\max\{r_1,\dots,r_p\}>t_s.
\end{equation}
Then \eqref{tt} can not hold.
We achieve this by showing that for a fixed integer $j\in \{1,\dots,p\}$,
if $s\geq n-\sum_{i=j}^{p}n_{w_i}+(p-j+1)n_{w_{j-1}}+1$
then \eqref{e-inequality} holds.

To warm up, we first show the $j=1$ case.
For $j=1$, since $s\geq n-\sum_{i=1}^pn_i+p+1=n_0+p+1$ we can assume that $s=n_0+p+l$
for a positive integer $l$. Hence,
\[
t_s=\Big\lfloor \frac{s-n_0-1}{p}\Big\rfloor
=\Big\lfloor \frac{p+l-1}{p}\Big\rfloor=1+\Big\lfloor \frac{l-1}{p}\Big\rfloor.
\]
To obtain the lower bound of $\max\{r_1,\dots,r_p\}$,
we can imagine that there are $s$ balls, and place these balls into the labelled
boxes $0,1,\dots,p$. The box zero can contain $n_0$ balls at most, the box 1 can contain $n_1$ balls at most,
the box 2 can contain $n_2$ balls at most, and so on.
The number $r_i$ is the number of the balls filled in the box $i$ for $i=1,\dots,p$.
If $s=n_0+p+l$, then
\[
\max\{r_1,\dots,r_p\}\geq \Big\lceil\frac{p+l}{p}\Big\rceil
=1+\Big\lceil\frac{l}{p}\Big\rceil>t_s=1+\Big\lfloor \frac{l-1}{p}\Big\rfloor.
\]
To achieve the aforementioned lower bound of $\max\{r_1, \dots, r_p\}$, we first place $n_0$ balls into box zero, ensuring that it is completely filled. Subsequently, we distribute the remaining $p + l$ balls evenly among the boxes 1 through $p$.

In general, for a given $j\in \{1,\dots,p\}$,
since $s\geq n-\sum_{i=j}^{p}n_{w_i}+(p-j+1)n_{w_{j-1}}+1$
we write
\begin{multline*}
s=n-\sum_{i=j}^{p}n_{w_i}+(p-j+1)n_{w_{j-1}}+l
=n_0+\sum_{i=1}^{j-2}n_{w_i}+(p-j+2)n_{w_{j-1}}+l\\
=n_0+pn_{w_1}+(p-1)(n_{w_2}-n_{w_1})+\cdots+(p-j+2)(n_{w_{j-1}}-n_{w_{j-2}})+l.
\end{multline*}
for a positive integer $l$.
In this case,
\begin{equation}\label{eq-maxr}
\max\{r_1,\dots,r_p\}\geq n_{w_{j-1}}+\Big\lceil\frac{l}{p-j+1}\Big\rceil.
\end{equation}
To attain the lower bound of $\max\{r_1,\dots,r_p\}$ for a general $j$ in \eqref{eq-maxr},
we begin by placing $n_0$ balls into the box zero, ensuring that it is completely filled.
Next, we place $pn_{w_1}$ balls among the boxes 1 through $p$, with each box receiving $n_{w_1}$ balls.
As a result, the box $w_1$ becomes completely filled.
Following this, we place $(p-1)(n_{w_2}-n_{w_1})$ balls evenly among the boxes $w_2$ through $w_p$,
ensuring that the box $w_2$ is also completely filled. We continue this process until the box $w_{j-1}$ is completely filled.
Finally, we distribute the remaining $l$ balls evenly among the boxes from $w_j$ to $w_p$.

Substitute $s=n-\sum_{i=j}^{p}n_{w_i}+(p-j+1)n_{w_{j-1}}+l$ into the expression for $t$ in \eqref{eq-tt} yields
$t_s=n_{w_{j-1}}+\big\lfloor \frac{l-1}{p-j+1}\big\rfloor$,
which is less than the lower bound of $\max\{r_1,\dots,r_p\}$
in \eqref{eq-maxr}.
\end{proof}

From the proof of Corollary~\ref{cor-3}, we find that for a given $j\in \{1,\dots,p\}$ and $n-\sum_{i=j}^{p}n_{w_i}+(p-j+1)n_{w_{j-1}}+1\leq s\leq
n-\sum_{i=j+1}^{p}n_{w_i}+(p-j)n_{w_j}$,
\[
\max\{r_1,\dots,r_p\}\geq n_{w_{j-1}}+\Big\lceil\frac{l}{p-j+1}\Big\rceil,
\]
where $l$ is a positive integer such that $l=s-n-\sum_{i=j}^{p}n_{w_i}+(p-j+1)n_{w_{j-1}}$.
While
\[
t_{s+1}=n_{w_{j-1}}+\Big\lfloor\frac{l}{p-j+1}\Big\rfloor.
\]
Therefore, we have
\[
\max\{r_1,\dots,r_p\}\geq t_{s+1}.
\]
The equality holds only when $l=(p-j+1)k$ for some $k\in \{1,\dots,n_{w_j}-n_{w_{j-1}}\}$. In this case, $\max\{r_1,\dots,r_p\}$ attains
its lower bound only when $r_0=n_0$, $r_{w_1}=n_{w_1}$,\dots,$r_{w_{j-1}}=n_{w_{j-1}}$,$r_{w_j}=\cdots=r_{w_p}=n_{w_{j-1}}+k$.
We conclude this result by the next corollary.
\begin{cor}\label{cor-rts}
Assume all the conditions of Corollary~\ref{cor-3},
\begin{equation}
\max\{r_1,\dots,r_p\}\geq t_{s+1},
\end{equation}
the equality holds only when $r_0=n_0$, $r_{w_1}=n_{w_1}$,\dots,$r_{w_{j-1}}=n_{w_{j-1}}$,$r_{w_j}=\cdots=r_{w_p}=n_{w_{j-1}}+k$.
Here $k$ is stated as above, and $w$ is any permutation of $\{1,\dots,p\}$ such that $(n_{w_1},\dots,n_{w_p})$ is increasing.
\end{cor}

\subsection{The vanishing properties}

In this subsection, we show that the rational function defined in \eqref{def-Q} below
has nice vanishing properties.

Let $b,c,d$ be nonnegative integers and $\varepsilon_{ij}$ be defined in \eqref{def-Tij}.
Denote
\begin{equation}\label{def-Q}
Q(d):=
\prod_{j=1}^{n}\frac{\big(qx_{j}/x_{0}\big)_b}
{\big(q^{-d}x_{0}/x_{j}\big)_{d}}
\prod_{1\leq i<j\leq n}
\big(x_{i}/x_{j}\big)_{c+\varepsilon_{ij}}\big(qx_{j}/x_{i}\big)_{c+\varepsilon_{ij}}.
\end{equation}
Since $D_n(a)$ is a polynomial in $q^a$ by Corollary~\ref{cor-poly},
we have
\begin{equation*}\label{e-QD}
\CT_{\x} Q(d)=D_n(-d)
\end{equation*}
for any integer $d$. Hence, we discuss the roots of $\CT\limits_x Q(d)$ instead of $D_n(a)$.
For any rational function $F$ of $x_0, x_1, \dots, x_n$ and $s$ an integer such that
$1\leq s\leq n$, and for sequences of integers $k = (k_1, k_2, \dots, k_s)$ and
$u = (u_1, u_2, \dots, u_s)$ let $E_{u,k}F$ be the result of
replacing $x_{u_i}$ in $F$ with $x_{u_s}q^{k_s-k_i}$ for
$i = 0,1,\dots, s-1$, where we set $u_0 = k_0 = 0$.
Then for $0 < u_1 < u_2 <\dots < u_s \leq n$ and $1\leq k_i\leq d$, we define
\begin{equation}\label{eq-Qrk}
Q(d\Mid u;k)
:=E_{u,k}\bigg(Q(d)\prod_{i=1}^{s}\Big(1-\frac{x_{0}}{x_{u_{i}}q^{k_{i}}}\Big)\bigg),
\end{equation}
and $Q(d\Mid u;k):=Q(d)$ for $s=0$.
Note that the product in \eqref{eq-Qrk} cancels all the factors in the denominator of $Q(d)$ that would be
taken to zero by $E_{u,k}$.
Applying the Gessel--Xin operation (mentioned in Section~\ref{sec-GX}) to $Q(d)$, we have
\begin{equation}\label{Q3-sum}
\CT_{\x}Q(d)=\sum_{s\in T}\sum_{\substack{1\leq u_1<\cdots<u_s\leq n\\1\leq k_1,\dots,k_s\leq d}}
\CT_{\x}Q(d\Mid u;k),
\end{equation}
where $T$ is a subset of $\{1,\dots,n\}$.

Let $U:=\{u_1,\dots,u_s\}$.
Then $Q(d\Mid u;k)$ can be written as
\begin{equation}\label{defi-Q}
Q(d\Mid u;k)=V\times H \times
\prod_{i=1}^s(q^{k_i-d})^{-1}_{d-k_i}(q)^{-1}_{k_i-1}
\prod_{\substack{1\leq i<j\leq n\\i,j\notin U}}\big(x_i/x_j\big)_{c+\varepsilon_{ij}}
(qx_{j}/x_{i}\big)_{c+\varepsilon_{ij}},
\end{equation}
where
\begin{equation}\label{def-V}
V=\prod_{i=1}^s(q^{1-k_i})_b
\prod_{1\leq i<j\leq s}(q^{k_j-k_i})_{c+\varepsilon_{u_iu_j}}(q^{k_i-k_j+1})_{c+\varepsilon_{u_iu_j}},
\end{equation}
and
\begin{equation}\label{eq-H}
H=\prod_{\substack{i=1\\ i\notin U}}^n
\frac{(q^{1-k_s}x_i/x_{u_s})_{b}}{\big(q^{k_s-d}x_{u_s}/x_i\big)_d}
\prod_{\substack{i=1\\ i\notin U}}^n\prod_{j=1}^s
\big(q^{k_j-k_s+\chi(i>u_j)}x_i/x_{u_s}\big)_{c+\varepsilon_{iu_j}}
\big(q^{k_s-k_j+\chi(u_j>i)}x_{u_s}/x_i\big)_{c+\varepsilon_{iu_j}}.
\end{equation}

Recall that $N_0=\{1,\dots,n_0\}$, $N_1=\{n_0+1,\dots,n_0+n_1\},\dots,
N_p=\{n_0+\cdots+n_{p-1}+1,\dots,n\}$ for positive integers $n_0,n_1,\dots,n_p$ such that $n_0+\cdots+n_p=n$. Define
\begin{equation}\label{r}
r_i:=|U\cap N_i|,
\end{equation}
for $i=0,\dots,p$.
The following properties of $Q(d\Mid u;k)$ are crucial in determining the roots of $D_n(a)$.
\begin{lem}\label{lem-Q}
Let $Q(d\Mid u;k)$, $r_i$ and $t_s$ be defined in \eqref{eq-Qrk}, \eqref{r}
and \eqref{def-t} respectively.
Then $Q(d\Mid u;k)$ has the following properties:
\begin{enumerate}
\item If $d\leq (s-1)c+b+t_s$, then $Q(d\Mid u;k)=0$;

\item If $d>sc+\sum_{i=1}^pr_i(n_i-r_i)/(n-s)$ and $s\neq n$, then
\begin{equation}\label{Q1}
\CT_{x_{u_s}}Q(d\Mid u;k)=
\begin{cases}\displaystyle
\sum_{\substack{u_s<u_{s+1}\leq n\\1\leq k_{s+1}\leq d}}
Q(d\Mid u_1,\dots,u_s,u_{s+1};k_1,\dots,k_s,k_{s+1}) \quad &\text{for $u_s<n$,}\\
0 \quad &\text{for $u_s=n$;}
\end{cases}
\end{equation}

\item If $sc+t_{s+1}+1\leq d\leq sc+\sum_{i=1}^pr_i(n_i-r_i)/(n-s)$ and $s\neq n$, then
\[
\CT_{\x} Q(d\Mid u;k)=0.
\]
\end{enumerate}
\end{lem}
\begin{proof}
(1)
Since $1\leq k_i\leq d$ for $i=1,\dots,s$, if $d$ satisfies the condition in (1) then
\[
1\leq k_i\leq (s-1)c+b+t_s \quad \text{for $i=1,\dots,s$}.
\]
By the definition of $r_i$, it is clear that $0\leq r_i\leq \min\{s,n_0\}$.
Then, at least one of the cases (i)--(iii) of
Corollary~\ref{cor-3} holds.

If (i) holds, then $Q(d\Mid u;k)$ has the factor
\[
E_{u,k}\Big[\big(qx_{u_i}/x_0\big)_b\Big]=\big(q^{1-k_i}\big)_b=0.
\]

If (ii) holds, then $Q(d\Mid u;k)$ has the factor
\[
E_{u,k}\Big[\big(x_{u_i}/x_{u_j}\big)_{c}\big(qx_{u_j}/x_{u_i}\big)_{c}\Big],
\]
which is equal to
\[
E_{u,k}\Big[q^{\binom{c+1}{2}}(-x_{u_j}/x_{u_i})^{c}\big(q^{-c}x_{u_i}/x_{u_j}\big)_{2c}\Big]
=q^{\binom{c+1}{2}}(-q^{k_i-k_j})^{c}\big(q^{k_j-k_i-c}\big)_{2c}=0.
\]

If (iii) holds, then $Q(d\Mid u;k)$ has the factor
\[
E_{u,k}\Big[\big(x_{u_i}/x_{u_j}\big)_{c+1}\big(qx_{u_j}/x_{u_i}\big)_{c+1}\Big],
\]
which is equal to
\[
E_{u,k}\Big[q^{\binom{c+2}{2}}(-x_{u_j}/x_{u_i})^{c+1}\big(q^{-c-1}x_{u_i}/x_{u_j}\big)_{2c+2}\Big]
=q^{\binom{c+2}{2}}(-q^{k_i-k_j})^{c+1}\big(q^{k_j-k_i-c-1}\big)_{2c+2}=0.
\]

(2)
We first show that $Q(d\Mid u;k)$ is of the form \eqref{e-defF} for $d>sc+\sum_{i=1}^pr_i(n_i-r_i)/(n-s)$ and $s\neq n$.

Recall that $U=\{u_1,u_2,\dots,u_s\}$. By the expression for $Q(d\Mid u;k)$ in \eqref{defi-Q},
the part which contributes to the degree in $x_{u_s}$
of the numerator of  $Q(d\Mid u;k)$ is
\[
\prod_{\substack{i=1\\ i\notin U}}^n\prod_{j=1}^s
\big(q^{k_s-k_j+\chi(u_j>i)}x_{u_s}/x_i\big)_{c+\varepsilon_{iu_j}},
\]
which has degree at most
\begin{equation*}
s(n-s)c+\sum_{i\notin U}\sum_{j=1}^s\varepsilon_{iu_j}
=s(n-s)c+\sum_{i=1}^pr_i(n_i-r_i).
\end{equation*}
The part which contributes to the degree in $x_{u_s}$ of the denominator of $Q(d\Mid u;k)$ is
\[
\prod_{\substack{i=1\\ i\notin U}}^n
\big(q^{k_s-d}x_{u_s}/x_i\big)_d,
\]
which has degree $(n-s)d$.
If $d>sc+\sum_{i=1}^pr_i(n_i-r_i)/(n-s)$ then
\[
s(n-s)c+\sum_{i=1}^pr_i(n_i-r_i)<(n-s)d.
\]
Thus $Q(d\Mid u;k)$ is of the form \eqref{e-defF}.
Applying Lemma~\ref{lem-almostprop} gives
\[
\CT_{x_{u_s}}Q(d\Mid u;k)=
\begin{cases}\displaystyle
\sum_{\substack{u_s<u_{s+1}\leq n\\1\leq k_{s+1}\leq d}}
Q(d\Mid u_1,\dots,u_s,u_{s+1};k_1,\dots,k_s,k_{s+1}) \quad &\text{for $u_s<n$,}\\
0 \quad &\text{for $u_s=n$.}
\end{cases}
\]

(3) The proof of Case (3) is lengthy, we give the proof in Subsection~\ref{subsec-case3} below.
\end{proof}

By Lemma~\ref{lem-Q}, we can prove Lemma~\ref{lem-roots}.
\begin{proof}[Proof of Lemma~\ref{lem-roots}]
Take $d=-a$.
Since $\CT\limits_{\x} Q(d)=D_n(-d)$, we determine the roots of $D_n(a)$ by showing that
\begin{equation}\label{fact}
\CT_{\x} Q(d)=0 \quad \text{for $d\in R$},
\end{equation}
where $R$ is defined in \eqref{def-roots}.
Then each $-d\in R$ is a root of $D_n(a)$.
We proceed the proof by induction on $n-s$ that
\begin{equation}\label{prop-roots13-1}
\CT_{\x} Q(d\Mid u;k)=0 \quad  \text{for $d\in R$}.
\end{equation}
The $s=0$ case of \eqref{prop-roots13-1} is just the fact \eqref{fact}.

If $s=n$ then $u_i$ must equal $i$
for $i=1,\dots ,n$ and thus
\[
Q(d\Mid u;k)=Q(d\Mid 1,2,\dots,n;k_1,k_2,\dots,k_n).
\]
Since $d\leq(n-1)c+b+n_{s_p}-1$ from the definition of $R$,
by the property (1) of Lemma~\ref{lem-Q} with $s=n$ we have
\[
Q(d\Mid 1,2,\dots,n;k_1,k_2,\dots,k_n)=0.
\]

Now suppose that $0\leq s<n$.
If $d\leq (s-1)c+b+t_s$, then
by the property (1) of Lemma~\ref{lem-Q} we have $Q(d\Mid u;k)=0$.
If $sc+t_{s+1}+1\leq d\leq sc+\sum_{i=1}^pr_i(n_i-r_i)/(n-s)$, then
by the property (3) of Lemma~\ref{lem-Q}
\[
\CT_{\x} Q(d\Mid u;k)=0.
\]
If $d>sc+\sum_{i=1}^pr_i(n_i-r_i)/(n-s)$, then by the property (2) of Lemma~\ref{lem-Q}
\[
\CT_{x_{u_s}}Q(d\Mid u;k)
=
\begin{cases}\displaystyle
\sum_{\substack{u_s<u_{s+1}\leq n\\1\leq k_{s+1}\leq d}}
Q(d\Mid u_1,\dots,u_s,u_{s+1};k_1,\dots,k_s,k_{s+1}) \quad &\text{for $u_s<n$,}\\
0 \quad &\text{for $u_s=n$.}
\end{cases}
\]
For $u_s<n$, applying $\CT\limits_{\x}$ to both sides of the above equation gives
\[
\CT_{\x}Q(d\Mid u;k)=\sum_{\substack{u_s<u_{s+1}\leq n\\1\leq k_{s+1}\leq d}}
\CT_{\x}Q(d\Mid u_1,\dots,u_s,u_{s+1};k_1,\dots,k_s,k_{s+1}).
\]
Each summand in the above is 0 by induction.

In conclusion, we show that $\CT\limits_{\x} Q(d\Mid u;k)=0$ for all $d\in R$,
completing the proof.
\end{proof}

\section{Case (3) of Lemma~\ref{lem-Q}}\label{sec-last}

In this section, we demonstrate that the rational function $Q(d \Mid u; k) $, as defined in \eqref{eq-Qrk}, can be transformed into a Laurent polynomial under certain conditions. By utilizing the vanishing coefficients outlined in Lemma~\ref{lem-vanishing}, we ultimately confirm that the constant term of
$Q(d \Mid u; k)$ vanishes under the conditions specified in Case (3) of Lemma~\ref{lem-Q}. These transformations effectively extend Gessel and Xin's proof of constant term identities via Laurent series.

It is important to note that we must address the scenarios in which Lemma~\ref{lem-almostprop} is not applicable. We define the degree of a variable (for instance, $z$) in a multivariable rational function as the degree of $z$ in the numerator minus the degree of $z$ in the denominator. Lemma~\ref{lem-almostprop} is limited to rational functions with negative degrees. In our previous work \cite{LXZ}, we extended this lemma to encompass nonpositive cases. More recently, we have further adapted the Gessel–Xin method to address positive cases \cite{XZ2023, Zhou23}. In fact, the primary results of this section concerning the case where $p=1$ were previously discussed in \cite[Section 9]{XZ2023}. Here, we extend these results to cover the general cases presented in this section.

\subsection{The Laurent polynomiality of $Q(d\Mid u;k)$}

In this subsection, we show that the rational function $Q(d\Mid u;k)$ is either 0, or can be written as a Laurent polynomial. This is the content of the next lemma.
\begin{lem}\label{lem-QLaurentpoly}
Let $(u_1,u_2,\dots,u_s)$ be an integer sequence such that $1\leq u_1<u_2<\cdots<u_s\leq n$ and define $U=\{u_1,u_2,\dots,u_s\}$.
We denote $r_i:=|U\cap N_i|$ as defined in \eqref{r} and let $t_s$ be defined as in \eqref{def-t} such that $\sum_{i=1}^pr_i(n_i-r_i)/(n-s)\geq t_{s+1}+1$.
If $1\leq d\leq sc+b+t_{s+1}$, then
$Q(d\Mid u;k)$ is either 0, or can be written as a Laurent polynomial of the form
\begin{equation}\label{Q-Laurent}
 x_{u_s}^l\prod_{\substack{i=1\\ i\notin U}}^n
\Big(p_i(x_i/x_{u_s})x_i^{d-sc-\sum_{j=1}^s\varepsilon_{iu_j}}\Big)
\prod_{\substack{1\leq i<j\leq n\\i,j\notin U}}\big(x_i/x_j\big)_{c+\varepsilon_{ij}}
(qx_{j}/x_{i}\big)_{c+\varepsilon_{ij}},
\end{equation}
where the $p_i(x_i/x_{u_s})$ are polynomials in $x_i/x_{u_s}$ and
\[
l=(n-s)(sc-d)+\sum_{i=1}^pr_i(n_i-r_i).
\]
\end{lem}

Recall that $Q(d\Mid u;k)$ can be written as \eqref{defi-Q}.
We can further write one of its factor $H$ (defined in \eqref{eq-H}) as
\begin{equation}\label{HH}
H=C\times\prod_{\substack{i=1\\ i\notin U}}^n
\frac{(q^{1-k_s}x_i/x_{u_s})_{b}}{(x_{u_s}/x_i)^d\big(q^{1-k_s}x_{i}/x_{u_s}\big)_d}\\
\prod_{\substack{i=1\\ i\notin U}}^n\prod_{j=1}^s
(x_{u_s}/x_i)^{c+\varepsilon_{iu_j}}
\big(q^{k_j-k_s-\chi(u_j>i)-\varepsilon_{iu_j}-c+1}x_i/x_{u_s}\big)_{2c+2\varepsilon_{iu_j}},
\end{equation}
where
\[
C=\prod_{\substack{i=1\\ i\notin U}}^n(-1)^dq^{\binom{d+1}{2}-dk_s}
\prod_{\substack{i=1\\ i\notin U}}^n\prod_{j=1}^s
(-1)^{c+\varepsilon_{iu_j}}
q^{(c+\varepsilon_{iu_j})\big(k_s-k_j+\chi(u_j>i)\big)+\binom{c+\varepsilon_{iu_j}+1}{2}}.
\]

Under certain conditions, the product $\prod_{\substack{i=1\\ i\notin U}}^n(q^{1-k_s}x_i/x_{u_s})_d$ in the denominator of $H$ is, in fact, a factor of the numerator of $H$. Consequently, $H$ can be expressed as a Laurent polynomial.
\begin{lem}\label{lem-Laurent}
For $s\in \{1,2,\dots,n-1\}$, let $r_i:=|U\cap N_i|$, $t_s$ and $R_i$ be defined as in \eqref{r}, \eqref{def-t} and \eqref{def-Ri} respectively. The $r_i$ satisfy
$\sum_{i=1}^pr_i(n_i-r_i)/(n-s)\geq t_{s+1}+1$, otherwise Case (3) of Lemma~\ref{lem-Q} does not occur.
Suppose $1\leq d\leq sc+b+t_{s+1}$, $k_i\in \{1,\dots,d\}$ for $i=1,\dots,s$, and
there exists a permutation $w\in\mathfrak{S}_s$ and nonnegative integers $\mathfrak{t}_1,\dots,\mathfrak{t}_s$
such that
\begin{equation}\label{kw1}
k_{w(1)}=b+\mathfrak{t}_1,
\end{equation}
and
\begin{equation}\label{kwj}
k_{w(j)}-k_{w(j-1)}=c+\chi\big(\{w(j),w(j-1)\}\subseteq R_i \ \text{for some $i=1,\dots,p$}\big)+\mathfrak{t}_j \quad \text{for $2\leq j\leq s$.}
\end{equation}
Here the $t_j$ satisfy
\begin{equation}\label{sumbound}
\max\{r_1,\dots,r_p\}\leq \sum_{j=1}^{s}\Big(\chi\big(\{w(j),w(j-1)\}\subseteq R_i \ \text{for some $i=1,\dots,p$}\big)+\mathfrak{t}_j\Big)\leq c+t_{s+1},
\end{equation}
$w(0):=0$, and $\mathfrak{t}_j>0$ if $w(j-1)<w(j)$ for $1\leq j\leq s$.
Then the rational function $H$ in \eqref{HH} can be written as the form
\begin{equation}\label{L-Laurent}
x_{u_s}^l\prod_{\substack{i=1\\ i\notin U}}^n
\Big(p_i(x_i/x_{u_s})x_i^{d-sc-\sum_{j=1}^s\varepsilon_{iu_j}}\Big),
\end{equation}
where
\[
l=(n-s)(sc-d)+\sum_{i=1}^pr_i(n_i-r_i),
\]
and the $p_i(x_i/x_{u_s})$ are polynomials in $x_i/x_{u_s}$.
\end{lem}
Note that the complex restrictions for the $k_i$ are simply a manifestation of Case (4) in Lemma~\ref{lem-key}.
\begin{proof}
 In \eqref{HH} we write $H$ as the form $\prod (x_{u_s}/x_i)^{\pm}\prod(1-q^lx_i/x_{u_s})^{\pm}$.
Therefore, we denote the range of $l$ as the following sets.
Let
\begin{align*}
S_0:=&\{1-k_s,2-k_s,\dots,d-k_s\},
\quad B:=\{1-k_s,2-k_2,\dots,b-k_s\}, \\
S_j:=&\{k_j-k_s-\chi(u_j>i)-\varepsilon_{iu_j}-c+1,\dots,k_j-k_s+c+\varepsilon_{iu_j}-\chi(u_j>i)\}
\end{align*}
for $j=1,\dots,s$ and a fixed $i\in \{1,\dots,n\}\setminus U$.
Assuming the conditions of the lemma, we demonstrate that the product
\[
\prod_{\substack{i=1\\ i\notin U}}^n \big(q^{1-k_s}x_i/x_{u_s}\big)_d
\]
in the denominator of $H$ is, in fact, a factor of its numerator. Consequently, $H$ can be expressed as a Laurent polynomial, rather than merely as a rational function.
To achieve this, it is sufficient to show that
\begin{equation}\label{set}
S_0\subseteq \bigcup_{j=1}^sS_j\bigcup B
\end{equation}
since the factors of $H$ in \eqref{HH} that require cancellation are of the form $\prod_j(1-q^jx_i/x_{u_s})$.
To obtain \eqref{set}, it suffices to prove
\begin{enumerate}
\item $b-k_s\geq k_{w(1)}-k_s-\chi(u_{w(1)}>i)-\varepsilon_{i u_{w(1)}}-c$, which is sufficient to show that
\[k_{w(1)}\leq b+c;\]
\item $k_{w(s)}-k_s+c+\varepsilon_{iu_j}-\chi(u_j>i)\geq d-k_s$, which is sufficient to show that
$k_{w(s)}\geq d+1-c$;
\item $k_{w(j-1)}+c+\varepsilon_{i u_{w(j-1)}}-\chi(u_{w(j-1)}>i)\geq k_{w(j)}-c-\varepsilon_{i u_{w(j)}}-\chi(u_{w(j)}>i)$ for $j=2,\dots,s$,
    which is
    \[
   k_{w(j)}-k_{w(j-1)}\leq 2c+\varepsilon_{i u_{w(j)}}+\varepsilon_{i u_{w(j-1)}}+\chi(u_{w(j)}>i)-\chi(u_{w(j-1)}>i).
    \]
\end{enumerate}

To prove (1), we need to find the upper bound for $k_{w(1)}$.
By \eqref{kwj} we have
\begin{equation}\label{sum}
\sum_{j=2}^s\big(k_{w(j)}-k_{w(j-1)}\big)=k_{w(s)}-k_{w(1)}
=(s-1)c+\sum_{j=2}^s\Big(\chi\big(\{w(j),w(j-1)\}\subseteq R_i \ \text{for some $i=1,\dots,p$}\big)+\mathfrak{t}_j\Big).
\end{equation}
It follows that
\begin{equation}\label{e-l92-2}
k_{w(1)}=k_{w(s)}-(s-1)c-\sum_{j=2}^s\Big(\chi\big(\{w(j),w(j-1)\}\subseteq R_i \ \text{for some $i=1,\dots,p$}\big)+\mathfrak{t}_j\Big).
\end{equation}
Since we can take $\chi\big(\{w(j),w(j-1)\}\subseteq R_i \ \text{for some $i=1,\dots,p$}\big)+\mathfrak{t}_j$ as the $e_j(w,r)$ in Proposition~\ref{prop-min}, by \eqref{lowerbound1} for any $l\in \{1,\dots,s\}$ we have
\begin{equation}\label{e-l92-1}
\sum_{\substack{j=1\\j\neq l}}^s\Big(\chi\big(\{w(j),w(j-1)\}\subseteq R_i \ \text{for some $i=1,\dots,p$}\big)+\mathfrak{t}_j\Big)\geq \max\{r_1,\dots,r_p\}-1.
\end{equation}
Hence, substituting \eqref{e-l92-1} with $l=1$ into \eqref{e-l92-2} gives
\begin{equation}\label{e-192-3}
k_{w(1)}\leq k_{w(s)}-(s-1)c-\max\{r_1,\dots,r_p\}+1.
\end{equation}
Since $k_{w(s)}\leq d\leq sc+b+t_{s+1}$,
\begin{equation}\label{e-Laurent-1}
k_{w(1)}\leq b+c+t_{s+1}-\max\{r_1,\dots,r_p\}+1.
\end{equation}
By Corollary~\ref{cor-rts},
for $t_{s+1}>0$ we have $\max\{r_1,\dots,r_p\}\geq t_{s+1}$. The equality $\max\{r_1,\dots,r_p\}=t_{s+1}$ holds only when
$r_0=n_0$, $r_{w_1}=n_{w_1}$,\dots,$r_{w_{j-1}}=n_{w_{j-1}}$,$r_{w_j}=\cdots=r_{w_p}=n_{w_{j-1}}+k$ for $w\in \mathfrak{S}_p$.
But in this case, by straightforward computation we can conclude $\sum_{i=1}^pr_i(n_i-r_i)/(n-s)=t_{s+1}$, which contradicts to the condition
$\sum_{i=1}^pr_i(n_i-r_i)/(n-s)\geq t_{s+1}+1$ in the lemma.
Hence,
\begin{equation}\label{e-Laurent-2}
\max\{r_1,\dots,r_p\}\geq t_{s+1}+1.
\end{equation}
Together with \eqref{e-Laurent-1} we have $k_{w(1)}\leq b+c$.
This is exactly the upper bound for $k_{w(1)}$ we require.
The only case remained is $t_{s+1}=0$.
In this case, $k_{w(s)}\leq d\leq sc+b$. Substituting this into \eqref{e-l92-2} yields
\[
k_{w(1)}\leq  c+b-\sum_{j=2}^s\Big(\chi\big(\{w(j),w(j-1)\}\subseteq R_i \ \text{for some $i=1,\dots,p$}\big)+\mathfrak{t}_j\Big)\leq c+b.
\]
This is also exactly the upper bound required for $k_{w(1)}$.
In conclusion, (1) holds for all the $s$.

To prove (2), we need to find the lower bound for $k_{w(s)}$.
By \eqref{sum} we have
\[
k_{w(s)}
=k_{w(1)}+(s-1)c+\sum_{j=2}^s\Big(\chi\big(\{w(j),w(j-1)\}\subseteq R_i \ \text{for some $i=1,\dots,p$}\big)+\mathfrak{t}_j\Big).
\]
Substituting \eqref{kw1} into the above equation gives
\begin{equation}\label{eq-kws}
k_{w(s)}
=b+(s-1)c+\sum_{j=1}^s\Big(\chi\big(\{w(j),w(j-1)\}\subseteq R_i \ \text{for some $i=1,\dots,p$}\big)+\mathfrak{t}_j\Big).
\end{equation}
If $t_{s+1}>0$ then by \eqref{sumbound} and \eqref{e-Laurent-2} we have
\[
k_{w(s)}
\geq b+(s-1)c+\max\{r_1,\dots,r_p\}\geq b+sc+t_{s+1}+1-c\geq d+1-c.
\]
Hence, (2) holds if $t_{s+1}>0$.
If $t_{s+1}=0$, then by utilizing the fact that $\mathfrak{t}_1\geq 1$ in \eqref{eq-kws}
we can derive the following inequality:
\[
k_{w(s)}
\geq b+(s-1)c+1
=b+sc-c+1\geq d-c+1.
\]
The last inequality is justified by the condition $sc+b\geq d$ when $t_{s+1}=0$.
Therefore, we conclude that (2) also holds for the case when $t_{s+1}=0$.

 If (3) fails for some $l\in \{2,\dots,s\}$, then
\[
    k_{w(l)}-k_{w(l-1)}\geq 2c+\varepsilon_{i u_{w(l)}}+\varepsilon_{i u_{w(l-1)}}+\chi(u_{w(l)}>i)-\chi(u_{w(l-1)}>i)+1.
\]
Together with \eqref{kwj} gives
\begin{multline}\label{kw-need}
k_{w(s)}-k_{w(1)}=\sum_{j=2}^s\big(k_{w(j)}-k_{w(j-1)}\big)
\geq sc+\sum_{\substack{j=2\\j\neq l}}^s\Big(\chi\big(\{w(j),w(j-1)\}\subseteq R_i \ \text{for some $i=1,\dots,p$}\big)+\mathfrak{t}_j\Big)\\
+\varepsilon_{i u_{w(l)}}+\varepsilon_{i u_{w(l-1)}}+\chi(u_{w(l)}>i)-\chi(u_{w(l-1)}>i)+1.
\end{multline}
Substituting \eqref{kw1} into the above equation, we have
\begin{multline*}
k_{w(s)}\geq sc+b+\sum_{\substack{j=1\\j\neq l}}^s\Big(\chi\big(\{w(j),w(j-1)\}\subseteq R_i \ \text{for some $i=1,\dots,p$}\big)+\mathfrak{t}_j\Big)\\
+\varepsilon_{i u_{w(l)}}+\varepsilon_{i u_{w(l-1)}}+\chi(u_{w(l)}>i)-\chi(u_{w(l-1)}>i)+1.
\end{multline*}
Substituting \eqref{e-l92-1} into the above equation and then using \eqref{e-Laurent-2}, we have
\begin{align}\label{e-3}
k_{w(s)}&\geq sc+b+\max\{r_1,\dots,r_p\}+\varepsilon_{i u_{w(l)}}+\varepsilon_{i u_{w(l-1)}}+\chi(u_{w(l)}>i)-\chi(u_{w(l-1)}>i)\\
&\geq sc+b+t_{s+1}+\varepsilon_{i u_{w(l)}}+\varepsilon_{i u_{w(l-1)}}+\chi(u_{w(l)}>i)-\chi(u_{w(l-1)}>i)+1\nonumber \\
&\geq d+\varepsilon_{i u_{w(l)}}+\varepsilon_{i u_{w(l-1)}}+\chi(u_{w(l)}>i)-\chi(u_{w(l-1)}>i)+1.\nonumber
\end{align}
If any of the equalities in \eqref{e-3} does not hold, then we can conclude that $k_{w(s)}>d$. This contradicts to
the fact that $k_{w(s)}\leq d$.
We discuss the following cases:
\begin{itemize}
\item If $t_{s+1}=0$, then $d\leq sc+b$. By \eqref{kw-need} and $k_{w(1)} \geq b+1$, we have
\[
k_{w(s)}\geq  sc+b+1>d.
\]
This is a contradiction.

\item If $t_{s+1}>0$ and all the equalities in \eqref{e-3} hold, then the next two conditions must be satisfied:
\begin{itemize}
    \item[(a)] $\varepsilon_{i u_{w(l)}}+\varepsilon_{i u_{w(l-1)}}+\chi(u_{w(l)}>i)-\chi(u_{w(l-1)}>i)+1=0$, which implies
        \begin{itemize}
        \item $\{i,u_{w(l)}\},\{i,u_{w(l-1)}\}\not\subset N_k$ for all $k=1,\dots,p$;
        \item $u_{w(l-1)}>i>u_{w(l)}$, which means that $w(l-1)>w(l)$ since $u_1 < u_2 < \cdots <u_s$.
    \end{itemize}
    Then it follows that $\{u_{w(l-1)},u_{w(l)}\}\not\subset N_k$ for all $k=1,\dots,p$.
    Otherwise if $\{u_{w(l-1)},u_{w(l)}\}\subset N_k$ for some $k$ then both $\{i,u_{w(l)}\}$ and $\{i,u_{w(l-1)}\}$ belong to $N_k$.
    Hence,
     $\{w(l),w(l-1)\}\not\subset R_k$ for all $k=1,\dots,p$. Because $\{u_i,u_j\}\subseteq N_k \iff \{i,j\}\subseteq R_k$ by their definitions.
\item[(b)]
\begin{equation}\label{e-3-0}
\sum_{\substack{j=1\\j\neq l}}^s\Big(\chi\big(\{w(j),w(j-1)\}\subseteq R_i \ \text{for some $i=1,\dots,p$}\big)+\mathfrak{t}_j\Big)=\max\{r_1,\dots,r_p\}-1.
\end{equation}
    \end{itemize}
Using the consequences that $w(l-1)>w(l)$ and $\{w(l),w(l-1)\}\not\subset R_k$ for all $k=1,\dots,p$ in (a),
we have
\begin{align*}
&\sum_{\substack{j=1\\j\neq l}}^s\Big(\chi\big(\{w(j),w(j-1)\}\subseteq R_i \ \text{for some $i=1,\dots,p$}\big)+\mathfrak{t}_j\Big)\\
&\geq \sum_{\substack{j=1\\j\neq l}}^s\Big(\chi\big(\{w(j),w(j-1)\}\subseteq R_i \ \text{for some $i=1,\dots,p$}\big)+\chi\big(w(j-1)<w(j)\big)\Big)\\
&=\sum_{j=1}^s\Big(\chi\big(\{w(j),w(j-1)\}\subseteq R_i \ \text{for some $i=1,\dots,p$}\big)+\chi\big(w(j-1)<w(j)\big)\Big)\\
&\geq \max\{r_1,\dots,r_p\}.
\end{align*}
The last inequality hods by \eqref{e-lowerbound}.
Then, \eqref{e-3-0} can not hold. Consequently, $k_{w(s)}>d$. This is a contradiction.
\end{itemize}
Therefore, we can conclude that (3) holds.

Since (1)-(3) hold,  we obtain \eqref{set}.
It follows that $H$ is a Laurent polynomial of the form
\[
\prod_{\substack{i=1\\ i\notin U}}^n
\frac{p_i(x_i/x_{u_s})}{(x_{u_s}/x_i)^d}
\prod_{\substack{i=1\\ i\notin U}}^n\prod_{j=1}^s
(x_{u_s}/x_i)^{c+\varepsilon_{iu_j}}
=x_{u_s}^l\prod_{\substack{i=1\\ i\notin U}}^n
\Big(p_i(x_i/x_{u_s})x_i^{d-sc-\sum_{j=1}^s\varepsilon_{iu_j}}\Big),
\]
where the $p_i(x_i/x_{u_s})$ are polynomials in $x_i/x_{u_s}$,
\[
l=(n-s)(sc-d)+\sum_{i=1}^pr_i(n_i-r_i)
\]
and $C$ can be put into the $p_i$'s.
\end{proof}

By Lemma~\ref{lem-Laurent} and Lemma~\ref{lem-key}, we can give a proof of Lemma~\ref{lem-QLaurentpoly}.
\begin{proof}[Proof of Lemma~\ref{lem-QLaurentpoly}]
Recall that $Q(d\Mid u;k)$ can be written as the form in \eqref{defi-Q}.
By Lemma~\ref{lem-key} with $t=c+t_{s+1}$, at least one of the following cases holds:
\begin{enumerate}
\item $1\leq k_i\leq b$ for some $i$ with $1\leq i\leq s$;
\item $-c\leq k_i-k_j\leq c-1$ for some $(i,j)$ such that $1\leq i<j\leq s$
and $\{i,j\}\nsubseteq R_l$ for all $l=1,2,\dots,p$;
\item $-c-1\leq k_{i}-k_{j}\leq c$ for some $(i,j)$ such that $1\leq i<j\leq s$
and $\{i,j\}\subseteq R_l$ for some $l\in \{1,2,\dots,p\}$;
\item there exists a permutation $w\in\mathfrak{S}_s$ and nonnegative integers $d_1,\dots,d_s$
such that
\begin{subequations}\label{e-k1}
\begin{equation}\label{e-k1-1}
k_{w(1)}=b+d_1,
\end{equation}
and
\begin{equation}\label{e-k1-2}
k_{w(j)}-k_{w(j-1)}=c+\chi\big(\{w(j),w(j-1)\}\subseteq R_i \ \text{for some $i=1,\dots,p$}\big)+d_j
 \quad \text{for $2\leq j\leq s$.}
\end{equation}
\end{subequations}
Here the $d_j$ satisfy
\begin{equation}\label{t}
\max\{r_1,\dots,r_p\}\leq \sum_{j=1}^{s}\Big(\chi\big(\{w(j),w(j-1)\}\subseteq R_i \ \text{for some $i=1,\dots,p$}\big)+d_j\Big)\leq c+t_{s+1},
\end{equation}
$w(0):=0$, and $d_j>0$ if $w(j-1)<w(j)$ for $1\leq j\leq s$.

\end{enumerate}

If one of the cases (1)-(3) holds, then $V$ defined in \eqref{def-V} is zero by a similar argument as that in the first part of the proof of Lemma~\ref{lem-Q}.
Since $V$ is a factor of $Q(d\Mid u;k)$, we can conclude that $Q(d\Mid u;k)=0$.

If Case (4) holds and $\sum_{i=1}^pr_i(n_i-r_i)/(n-s)\geq t_{s+1}+1$,
then $Q(d\Mid u;k)$ can be written as the form of \eqref{Q-Laurent} by Lemma~\ref{lem-Laurent}.
Here the extra coefficients can be put into the $p_i$'s.
\end{proof}

\subsection{Proof for Case (3) of Lemma~\ref{lem-Q}}\label{subsec-case3}

In this subsection, we provide a proof for Case (3) of Lemma~\ref{lem-Q}.
That is, under the conditions of Lemma~\ref{lem-Q},
if $s\neq n$ and
$sc+t_{s+1}+1\leq d\leq sc+\sum_{i=1}^pr_i(n_i-r_i)/(n-s)$,
then $\CT\limits_x Q(d\Mid u;k)=0.$
\begin{proof}[Proof for Case (3) of Lemma~\ref{lem-Q}]
We can apply  Lemma~\ref{lem-QLaurentpoly} to obtain that
$Q(d\Mid u;k)$ is either 0, or can be written as a Laurent polynomial of the form
\begin{equation*}
L:=x_{u_s}^l\prod_{\substack{i=1\\ i\notin U}}^n
\Big(p_i(x_i/x_{u_s})x_i^{d-sc-\sum_{j=1}^s\varepsilon_{iu_j}}\Big)
\prod_{\substack{1\leq i<j\leq n\\i,j\notin U}}\big(x_i/x_j\big)_{c+\varepsilon_{ij}}
(qx_{j}/x_{i}\big)_{c+\varepsilon_{ij}},
\end{equation*}
where the $p_i(x_i/x_{u_s})$ are polynomials in $x_i/x_{u_s}$, $l=(n-s)(sc-d)+\sum_{i=1}^pr_i(n_i-r_i)$ and $U=\{u_1,\dots,u_s\}$.
Since $d\leq sc+\sum_{i=1}^pr_i(n_i-r_i)/(n-s)$, it follows that $l$ is a nonnegative integer.
By taking the constant term of $L$ with respect to $x_{u_s}$,
we can write $\CT\limits_{x_{u_s}}L$ as a finite sum of the form
\begin{align}
P:&=c\times\frac{\prod_{\substack{i=1\\ i\notin U}}^{n_0}x_i^{d-sc}}
{\prod_{\substack{i=n_0+1\\ i\notin U}}^nx_i^{sc+\sum_{j=1}^s\varepsilon_{iu_j}-d}}
\prod_{\substack{i=1\\ i\notin U}}^n x_i^{t_i}\prod_{\substack{1\leq i<j\leq n\\i,j\notin U}}\big(x_i/x_j\big)_{c+\varepsilon_{ij}}
(qx_{j}/x_{i}\big)_{c+\varepsilon_{ij}}\nonumber \\
&=c\times\frac{\prod_{\substack{i=1\\ i\notin U}}^{n_0}x_i}
{\prod_{u=1}^p\prod_{\substack{v\in N_u\\v\notin U}}x_v^{sc+r_u-d}}
\prod_{\substack{i=1\\ i\notin U}}^{n_0}x_i^{d-sc-1}\prod_{\substack{i=1\\ i\notin U}}^n x_i^{t_i}\prod_{\substack{1\leq i<j\leq n\\i,j\notin U}}\big(x_i/x_j\big)_{c+\varepsilon_{ij}}
(qx_{j}/x_{i}\big)_{c+\varepsilon_{ij}},\label{LP1}
\end{align}
where $c\in K$ and the $t_i$ are nonnegative integers such that
$\sum_{\substack{i=1\\i\notin U}}^nt_i=l$.
Here $d-sc-1\geq 0$ since $d\geq sc+t_{s+1}+1$.
We will show that all forms of $\CT\limits_{\x}P=0$.

It is not hard to see that \eqref{LP1} is of the form \eqref{V0-2} by variable substitutions.
We are going to prove $\CT\limits_{\x}P=0$ by Lemma~\ref{lem-vanishing}.
To achieve this, it is sufficient to show that the sum of the distinct powers of the $x_v$ in the denominator of \eqref{LP1}
\begin{equation}\label{e-last}
\sum_{u=1}^p(sc+r_u-d)\leq n_0-r_0-1.
\end{equation}
We can further write the left-hand side of \eqref{e-last} as
\begin{equation}\label{e-last1}
p(sc-d)+r_1+\cdots+r_p\leq -p(t_{s+1}+1)+r_1+\cdots+r_p
\end{equation}
by $d\geq sc+t_{s+1}+1$.
Hence, if we can prove that
\[
-p(t_{s+1}+1)\leq n_0-r_0-\cdots-r_p-1=n_0-s-1,
\]
then \eqref{e-last} holds. This is confirmed by the next lemma.
\end{proof}

\begin{lem}
For a positive integer $s$, let $n_0,\dots,n_p$ be positive integers such that
$n_0+\cdots+n_p=n$.
Take
\begin{equation}\label{ts}
t_{s}=\begin{cases}
0 \quad &\text{if $1\leq s\leq n-\sum_{i=1}^p n_i+p$},\\
\big\lfloor\frac{s-n_0-1}{p} \big\rfloor \quad &\text{if $n-\sum_{i=1}^p n_i+p+1\leq s
\leq n-\sum_{i=2}^pn_{s_i}+(p-1)n_{s_1}$},\\
\big\lfloor\frac{s-n_0-n_{s_1}-1}{p-1} \big\rfloor \quad
&\text{if $n-\sum_{i=2}^pn_{s_i}+(p-1)n_{s_1}+1\leq s
\leq n-\sum_{i=3}^pn_{s_i}+(p-2)n_{s_2}$},\\
\quad \vdots\\
s-n_0-n_{s_1}-\cdots-n_{s_{p-1}}-1 \quad &\text{if
$n-n_{s_p}+n_{s_{p-1}}+1\leq s
\leq n$,}
\end{cases}
\end{equation}
where $(n_{s_1},\dots,n_{s_p})$ is a permutation of $(n_1,\dots,n_p)$ in an increasing order.
Then
\begin{equation}\label{e-last0}
-p(t_{s}+1)\leq n_0-s.
\end{equation}
\end{lem}
Note that the definition of $t_s$ here is exactly the same as \eqref{def-t}.
\begin{proof}
To warm up, we first prove the $t_s=0$ case for $1\leq s\leq n-\sum_{i=1}^pn_i+p$.
In this case,
\[
s\leq n-\sum_{i=1}^pn_i+p=n_0+p.
\]
It follows that
\[
n_0-s\geq -p,
\]
which is the  $t_s=0$ case of \eqref{e-last0}.

As in the proof of Lemma~\ref{cor-3}, for the $t_s>0$ cases in \eqref{ts}, we can summarize the formulas
for $t_s$ according to the rows $j=1,\dots,p$:
\begin{equation}\label{eq-t}
t_s=\Big\lfloor\frac{s-n_0-\sum_{i=1}^{j-1}n_{s_i}-1}{p-j+1} \Big\rfloor
\end{equation}
for
\begin{equation}\label{e-s}
n-\sum_{i=j}^{p}n_{s_i}+(p-j+1)n_{s_{j-1}}+1\leq s\leq
n-\sum_{i=j}^{p}n_{s_i}+(p-j+1)n_{s_j}.
\end{equation}
Here for the $j=1$ case we take $n_{s_0}=1$.
Note that there are $p+1$ rows in \eqref{ts}. We identify the first row as the row 0.
For a fixed $j\in \{1,\dots,p\}$, there are $(p-j+1)(n_{s_j}-n_{s_{j-1}})$ such integers $s$ that satisfy \eqref{e-s}.
Hence, we can further assume that $n_{s_p}>n_{s_{p-1}}>\cdots>n_{s_1}>1$. Otherwise, the row $j$ in \eqref{ts} actually does not exist.

Having the above definitions, to prove \eqref{e-last0}, it suffices to prove that for a given $j\in \{1,\dots,p\}$,
\begin{equation}\label{ineq-last}
-p\Big(\Big\lfloor\frac{s-n_0-\sum_{i=1}^{j-1}n_{s_i}-1}{p-j+1} \Big\rfloor+1\Big)\leq n_0-s,
\end{equation}
where $s$ is in the range of \eqref{e-s}.

We further divide the range of $s$ in \eqref{e-s} into $n_{s_j}-n_{s_{j-1}}$ parts:
\begin{equation}\label{e-last-s}
n-\sum_{i=j}^{p}n_{s_i}+(p-j+1)n_{s_{j-1}}+(p-j+1)(k-1)+1\leq s\leq
n-\sum_{i=j}^{p}n_{s_i}+(p-j+1)n_{s_{j-1}}+(p-j+1)k
\end{equation}
for $k=1,\dots,n_{s_j}-n_{s_{j-1}}$. Note that the $j$ here is a fixed integer in $\{1,\dots,p\}$ and $n_{s_j}-n_{s_{j-1}}\geq 1$ by our assumption.
It is easy to see that if the $s$ is in the range of \eqref{e-last-s}, then the left-hand side of \eqref{ineq-last}
\[
-p\Big(\Big\lfloor\frac{s-n_0-\sum_{i=1}^{j-1}n_{s_i}-1}{p-j+1} \Big\rfloor+1\Big)
=-p(n_{s_{j-1}}+k).
\]
Using the upper bound for $s$ in \eqref{e-last-s}, we have
\begin{align*}
n_0-s&\geq n_0-n+\sum_{i=j}^{p}n_{s_i}-(p-j+1)n_{s_{j-1}}-(p-j+1)k\\
&=-p(n_{s_{j-1}}+k)+n_0-n+\sum_{i=j}^{p}n_{s_i}+(j-1)n_{s_{j-1}}+(j-1)k\\
&\geq -p(n_{s_{j-1}}+k)+n_0-n+\sum_{i=j}^{p}n_{s_i}+n_{s_1}+\cdots+n_{s_{j-1}}+(j-1)k\\
&= -p(n_{s_{j-1}}+k)+(j-1)k\geq -p(n_{s_{j-1}}+k).
\end{align*}
The last inequality in the above holds by the fact that $j,k\geq 1$. Therefore, we can conclude that \eqref{ineq-last} holds.
\end{proof}

\subsection*{Acknowledgements}

This work was supported by the
National Natural Science Foundation of China (No. 12171487).

\end{document}